\documentclass[11pt]{amsart}
\usepackage{geometry}                
\usepackage[parfill]{parskip}    
\usepackage{graphicx}
\usepackage{amssymb}
\usepackage{epstopdf}
\usepackage{hyperref}
\usepackage[noabbrev]{cleveref}
\usepackage{amsfonts,amsmath,amssymb,amsthm}
\usepackage{bm}
\usepackage{todonotes}

\definecolor{skyblue}{rgb}{0.85,0.85,1}

\theoremstyle{plain}
          \newtheorem{theorem}{Theorem}[section]
          \newtheorem{proposition}[theorem]{Proposition}
          \newtheorem{lemma}[theorem]{Lemma}

\theoremstyle{definition}
          \newtheorem{definition}[theorem]{Definition}
	  \newtheorem{remark}[theorem]{Remark}

\newcommand{\bp}{\bm{p}}
\newcommand{\bq}{\bm{q}}
\newcommand{\ba}{\bm{a}}

\newcommand{\q}{\theta}
\newcommand{\ve}{\varepsilon}

\newcommand{\cS}{\mathcal{S}}

\newcommand{\C}{\mathbb{C}}
\newcommand{\R}{\mathbb{R}}
\newcommand{\T}{\mathbb{T}}
\newcommand{\Z}{\mathbb{Z}}

\begin{document}

\title[Euler, Hill, Evans]{An Evans function for the linearised 2D Euler equations using  Hill's determinant}

\author{Holger~R.~Dullin}
\address{School of Mathematics and Statistics, The University of Sydney, Australia}
\email{holger.dullin@sydney.edu.au}

\author{Robert Marangell}
\address{School of Mathematics and Statistics, The University of Sydney, Australia}
\email{robert.marangell@sydney.edu.au}

\begin{abstract}
We study the point spectrum of the linearisation  of  Euler's equation for the ideal fluid on the torus
about a shear flow.
By separation of variables the problem is reduced to the spectral theory of a complex Hill's equation.
Using  Hill's determinant an Evans function of the original Euler equation is constructed.
The Evans function allows us to completely characterise the point spectrum of the linearisation,
and to count the isolated eigenvalues with non-zero real part. We prove that the number of discrete eigenvalues of he linearised operator for a specific shear flow is exactly twice the number of non-zero integer lattice points inside the so-called unstable disk. 
\end{abstract}

\maketitle

\section{Introduction}

Euler's fluid equations describe the flow of an incompressible inviscid Newtonian fluid.
It is a non-linear partial differential equation that is a limit of the incompressible Navier-Stokes equations for vanishing viscosity. In the stream function formulation the equations read 
\begin{equation}\label{eq:euler}
\nabla^2 \Psi_t = \Psi_x(\nabla^2 \Psi)_y - \Psi_y(\nabla^2 \Psi)_x\,,
\end{equation}
where $\Psi(x,y,t)$ is the stream function of the velocity field, that is $u = \Psi_y $ and $v = -\Psi_x $ are the $x$ and $y$ components of the velocity of an incompressible ideal fluid flow, and subscripts denote partial derivatives. 

The simplest solutions of \cref{eq:euler} are steady solutions for which the velocity field is constant in time.
In a shear flow adjacent layers of fluid move parallel to each other, with possibly different speeds. 
Given such an equilibrium solution its stability is studied by linearising the Euler equations about this particular solution.
When considering the equation with periodic boundary conditions in both directions it is natural to consider
the linearised equations in Fourier space. The linear operator typically has continuous spectrum on the imaginary axis,
and may have discrete spectrum outside the imaginary axis (see for example \cite{Li00,LS03,LLS04} and the references therein). Due to the Hamiltonian symmetry, any spectrum with non-zero real part leads to linear instability through exponential growth.

We study the linearisation of \cref{eq:euler} about the steady state with stream function $\psi = \cos( p_1 x + p_2 y)$ for 
co-prime integers $\bp = (p_1, p_2)$. 
For the linearised Euler equations in Fourier space, an integer lattice point $\ba=(a_1,a_2) \in \Z^2$ represent the spatial Fourier mode
$\exp(i( a_1 x + a_2 y))$.
The linear operator obtained from linearising \cref{eq:euler} about $\psi$ block-diagonalises into so-called classes.
A class is an infinite collection of modes whose integer lattice points $\ba$ are located on a line with direction vector~$\bp$ \cite{Li00,LLS04,DMW16}.
Different classes correspond to lines with different (signed) distance to the origin.

A theorem of Li \cite{Li00} shows that classes which do not have an integer lattice point inside the `unstable disk' of radius $\sqrt{p_1^2+p_2^2}$ can not contribute to instability. A theorem of Latushkin, Li and Stanislavova \cite{LLS04} shows that twice the number of non-zero integer lattice points inside the unstable disk is an upper bound for  the number of discrete eigenvalues with non-zero real part.
In this paper we show that this upper bound is sharp, so that the number of discrete eigenvalues is exactly equal to twice the number of  non-zero integer lattice points inside the unstable disk. 

This result has been difficult to achieve because the discrete spectrum typically contains complex eigenvalues.
It has been shown by Worthington, Dullin and Marangell \cite{DMW16}, and Latushkin and Vasuvedan \cite{LV19} that classes which contain a single integer lattice point inside the unstable disk correspond to a pair of real eigenvalues. 
In this manuscript we provide another proof of this fact, and we show that classes containing a pair of integer lattice points inside the unstable disk correspond to either two pairs of real eigenvalues  or a complex quartet. 

These new results about the (in)stability of the Euler equations are obtained through an approach different from \cite{Li00,DMW16,LV19}. 
Instead of first going into Fourier space, then linearising the equation, and then block-diagonalising the operator, 
we first linearise the PDE, then separate variables, and then consider the resulting boundary value problem in (a different) Fourier space. 
The separation step is somewhat unusual because it uses a non-orthogonal coordinate system in order to preserve 
the periodicity. The resulting boundary value problem/dispersion relation is a Hill's equation with quasi-periodic boundary conditions.
The spectral theory of this equation is well developed and is based on Hill's determinant, which is a Fredholm determinant. 
The Hill's determinant is used to write down an Evans function for a class. In the final step all relevant Evans functions coming
from Hill's equation are combined into a single Evans function of the linearised 2D Euler equations. 
With the help of this Evans function we are able to prove that the upper bound from \cite{LLS04}  is actually always attained. 
The main ingredient in the proof of our result is the ability to differentiate the Hill's determinant and hence the Evans function at vanishing spectral parameter. This allows us to treat questions about the number of unstable eigenvalues as a kind of bifurcation problem. 

{\color{black} Our periodic Evans function for a class is a particular example of a periodic Evans function associated with Bloch-wave eigenvalue problems of PDEs in 1+1 dimensions. There are many other specific examples of such Evans functions constructed in the last 15 years or so, for various PDEs of note, including the Kuramoto-Sivashinsky equation \cite{BNRZ12}, the Kawahara equation \cite{TDK2018}, the sine and Klein-Gordon equations \cite{CM20,JMMP13,JMMP14}, as well as linear and nonlinear Schr\"odinger equations \cite{BR05,CM20b,GH07,JLM13,LBJM19} to name a few. For a description of periodic Evans functions along the lines of \cite{Gar93}, as well as for additional applications of the Evans function to periodic problems, see \cite{KP13} and the references therein. 

Our Evans function differs from the aforementioned references in that, rather than linearising around a periodic solution to a 1+1 dimensional travelling wave and considering arbitrary longitudinal perturbations, our eigenvalue problem, see Proposition~\ref{prop:Hill}, comes from the reduction of a 2+1 dimensional dispersion relation that is considering  periodic 
perturbations in both directions. After separation of variables the temporal spectral parameter $\lambda$ (or $c$) is now a parameter within the potential of the Hill's equation, and what would typically be a range of Floquet exponents is reduced to a single Bloch mode, as it would be in the 1+1 dimensional case of periodic perturbations, but now the `eigenvalue' parameter (what we call $\mu = d^2$), is a function of the Fourier mode of the periodic perturbation in the transverse direction. The upshot is that we can construct an Evans function for a 2+1 dimensional problem by multiplying together the Evans functions for the relevant classes, with Li's theorem \cite{Li00} ensuring that the relevant product is finite. The lack of continuity in the Floquet exponents considered means that we cannot exploit the analyticity of the periodic Evans function near the origin for a stability index, as in \cite{BJK2011,BJK2014,BR05,JMMP13,JMMP14}. Instead we employ the Floquet-Fourier-Hill method of \cite{DK06} (and \cite{MW66}) in order to relate the Hill discriminant to the parameters in our original equation. 

Evans functions for systems with more than one independent spatial variable are not common. To the best of our knowledge they have only been determined numerically twice: first, to examine the stability (or lack thereof) in large amplitude planar Navier-Stokes shocks \cite{HLZ17} and then in the case of planar viscous strong detonation waves \cite{BHLL18}. 
The first proofs of instability for the equilibria considered here were given in \cite{DMW16} and \cite{LV19,DLMVW19}, but the Evans function was not used in this approach. Here we use an Evans function for this $2+1$ dimensional problem to give a precise count of all unstable eigenvalues, including complex eigenvalues. 
Extensions to non-planar domains \cite{DW16} and to $3+1$ dimensions \cite{DW19b,DMW19} have not been considered in the framework of Evans functions, and we hope to return to these problems in the future.
}

The structure of the paper is as follows.
In \cref{sec:separate} we linearise the Euler equations about a shear flow solution, separate variables and derive a related Hill's equation for the spectral problem. 
In \cref{sec:hill,sec:det} we consider the spectrum of the Hill's equation via the Hill determinant and derive the Evans function for a single class of the linearised Euler equations on the torus. 
In \cref{sec:evans} we prove that we indeed have an Evans function for this class. 
In \cref{sec:euler} we extend the Evans function for a class to an Evans function for the linearised Euler equations. 

\section{Separation of the linearised Euler equations}\label{sec:separate}

Consider \cref{eq:euler} on the torus $\T^2 = [-\pi,\pi] \times [-\pi,\pi]$.
A simple steady shear flow in the $x$-direction is given by $\Psi(x, y) = \cos y$, such that $u = -\sin y$ and $v = 0$.
More generally,  a doubly periodic steady state solution to \cref{eq:euler} is given by $\Psi(x,y) = \psi(\eta)$, where $\psi(\eta) = \psi( \eta + 2\pi)$ is a real periodic function 
of the single variable $\eta = p_1 x + p_2 y$  with $p_1, p_2 \in \Z$  relatively prime integers. 
Such a solution $\psi(\eta)$ is steady state solution to \cref{eq:euler}. 

This particular solution is best described in a new coordinate system $(\xi, \eta)$ with $\eta$ as defined before, 
and $\xi = q_1 x + q_2 y$ with relatively prime integers $q_1, q_2 \in \Z$, i.e.\ $\gcd(q_1, q_2) = 1$.
When $\gcd(p_1, p_2) = 1$ then there exist relatively prime integers $q_1, q_2 \in \Z$ with $p_2 q_1 - p_1 q_2 = 1$ (see, e.g., \cite[Thm25]{HardyWright79}), such that the vectors
$\bq = (q_1, q_2)$ and $\bp = (p_1, p_2)$ are a basis of the lattice $\Z^2$. The vector $\bq$ is unique up to adding multiples of $\bp$.
We choose the unique $\bq$ that makes the angle between $\bp$ and $\bq$ as large as possible, equivalently $|\bp \cdot \bq|$ as small as possible.
For given fixed $\bp$  this basis of $\Z^2$  is as close as possible to an orthogonal basis. 
The linear transformation from $(x,y)$ to $(\xi, \eta)$ preserves periodicity and volume.
The Laplacian in the $(\xi, \eta)$ coordinates reads
\[
    \nabla^2 = q^2 \partial_\xi^2 + 2 \bp \cdot \bq \, \partial_\xi \partial_\eta + p^2 \partial_\eta^2 \,,
\]
where $p^2 := p_1^2 + p_2^2$ and $q^2:= q_1^2 + q_2^2$. Note that the choice of $\bq$ implies that $|\bp \cdot \bq| / p^2 \le 1/2$.
The mixed term in the Laplacian appears because the chosen coordinate system is not in general orthogonal. 
For our purposes it is crucial to preserve periodicity and volume, so that instead of orthogonality we just
require that the  determinant of the linear map from $(x,y)$ to $(\xi, \eta)$ has determinant~1.

Now consider the linearisation of \cref{eq:euler} about $\psi(p_1 x + p_2 y)$ written in the new coordinate system.
The perturbed solution is $\psi(\eta) + \ve\phi(\xi,\eta,t)$ and considering only the first order terms in $\ve$ the linear governing equation for the perturbation $\phi$ is
\begin{equation}\label{eq:lin}
    \nabla^2 \phi_t = p^2 \psi''' \phi_\xi  - \psi' \nabla^2 \phi_\xi 
\end{equation}
where $\frac{d}{d\eta}$ is denoted by $'$. 
Two additional terms of the linearised equation vanish because $\psi_\xi = 0$. 
This is a linear PDE for $\phi(\xi, \eta, t)$ with coefficients that depend periodically on $\eta$ and periodic boundary conditions $\phi( \xi + 2\pi, \eta ) = \phi(\xi, \eta+ 2 \pi) = \phi(\xi, \eta)$.

Consider a single mode $\phi = f(\eta) \exp( i k \xi + \lambda t))$ with wave number $k \in \Z$ and temporal eigenvalue $\lambda \in \C$ where
 $f(\eta)$ is a $2\pi$-periodic function. Then $\nabla^2 \phi = ( -k^2 q^2 f + 2 i k \bp \cdot \bq f' + p^2 f'') \phi / f$
 and
 \cref{eq:lin} separates into
\begin{equation}\label{eq:sep}
    f'' + \frac{2 i k  \bp \cdot \bq}{p^2} f' + \left( \frac{k\psi'''(\eta)}{i\lambda - k\psi'(\eta)} - \frac{k^2q^2}{p^2}\right)f = 0 \,.
\end{equation}
For $\bp \cdot \bq = 0$ and $p^2 = q^2 = 1$ this reduces to the usual equation for linearisation about a 2D shear flow, 
see, e.g., \cite{Acheson90}, except that there the velocity $U = \psi_\eta$ is used to write the coefficient.
The term proportional to $f'$ in this second order linear equation can be transformed away by redefining the independent variable $f$,
at the expense of introducing quasiperiodic boundary conditions.
Defining $\q = k  \bp \cdot \bq/p^2 \bmod 1$ and $f(\eta) = e^{-i\theta \eta} g(\eta)$ we arrive at
\begin{proposition} \label{prop:Hill}
The linearisation of Euler's equation for the ideal incompressible fluid flow \eqref{eq:euler} on the torus $(x,y) \in \T^2$ around the steady state with periodic stream function $\psi(\eta)$ where $\eta = p_1 x + p_2 y$ leads to the Hill's equation
\begin{equation}\label{eq:hillgen}
    g'' + Q(\eta) g = \mu g, \quad Q(\eta) :=  \frac{\psi'''}{c - \psi'} 
\end{equation}
where 
$\mu := k^2/p^4$ and $c := i \lambda/k$, with quasi-periodic boundary conditions 
\begin{equation}\label{eq:bound}
\begin{split}
g(2\pi) & = \exp( 2\pi i \theta) g(0),  \\ 
g'(2\pi) & = \exp( 2\pi i \theta) g'(0) \,.
\end{split}
\end{equation}
\end{proposition}

We have thus rewritten the dispersion relation for the linearised equation \eqref{eq:lin} on the torus as the quasi-periodic boundary value problem \cref{eq:hillgen} with boundary conditions \eqref{eq:bound}. 
There are temporal eigenvalue $\lambda = -i c k$ of the linearised Euler equation (about the shear flow $\psi(\eta)$)
when there exists a quasi-periodic solution to \cref{eq:hillgen} satisfying the boundary condition in \cref{eq:bound}. 

The spatial mode $\exp( i ( l\eta + k \xi))$ written in the original coordinates is $\exp( i ( l p_1 + k q_1) x + ( l p_2 + k q_2) y) = \exp( i ( a_1 x + a_2 y))$, 
such that $\ba = k \bq + l \bp \in \Z^2$ is the Fourier mode in the original coordinates.
It is convenient to introduce yet another coordinates system in dual Fourier-mode space with orthogonal basis $(\bp_\perp, \bp)$. 
The vector $\bp_\perp = -J \bp$ where $J$ is a rotation by $\pi/2$.
The vector $\ba$ in this coordinate system has coordinates denoted by $(d, \q)$, such that
$\ba = k \bq + l\bp =  d \bp_\perp + \q \bp$ implies
\begin{equation} \label{eq:dtheta}
        d = \frac{ k}{p^2}, \quad \q = k \frac{ \bp \cdot \bq}{p^2 }  \,.
\end{equation}
This is obtained by taking scalar products with $\bp$ and $\bp_\perp$ and using $\bq \cdot \bp_\perp = 1$.
Notice that in the boundary condition \eqref{eq:bound} any integer part in $\theta$ is irrelevant, and 
this is the reason that $\q$ can be considered as an angle $\mod 1$.
In this coordinate system the linearisation of  \cref{eq:euler} about $\psi(p_1 x + p_2 y)$ directly leads to 
\cref{eq:hillgen}.
The quantity $\q$ measures the deviation of the basis $(\bq, \bp)$ from orthogonality.
Conversely, when using the orthogonal basis $(\bp_\perp, \bp)$ the quantity $\q$ determines the 
quasi-periodic boundary conditions.

For later reference we translate Li's unstable disk theorem \cite{Li00} into our setting.
Li studied the case $\psi(\eta) = \cos(\eta)$ and decomposed the Fourier modes $\ba = k \bq + l \bp$ into subsets called ``classes'' 
with fixed integer $k$ and varying integer $l$. The unstable disk theorem states that the
class with given $k$ has no discrete spectrum when all modes $\ba$ in this class are 
outside the closed disk centred at the origin with radius $p = |\bp| = \sqrt{ p_1^2 + p_2^2}$.
The mode in class $k$ that is closest to the origin has the smallest projection onto the $\bp$-direction. 
Hence $k \bq \cdot \bp + l \bp \cdot \bp$ is minimal and dividing by $p^2 = \bp \cdot \bp$ we see that 
this occurs when $\q - l$ is minimal. 
Thus we  choose the angle $\q$ to be such that $ \q \in (-1/2, 1/2]$.
This determines a particular representative of the class with given $k$ for which $l$ is uniquely determined.
Thus in our notation Li's unstable disk theorem \cite{Li00} reads
\begin{theorem} \label{thm:disk}
Consider the case where $\psi(\eta) = \cos(\eta)$. If $\q^2  + d^2 > 1$  then there is no discrete eigenvalue $c$ of the linearised Euler's equations
and the solution is stable.
Equivalently, if $\mu = d^2  > 1 - \theta^2$ then there is no $\q$-quasiperiodic spectrum  of Hill's equation \eqref{eq:hillgen} for any $c$. 
\end{theorem}

The spectral parameter $\mu$ of Hill's equation \eqref{eq:hillgen} satisfies $\mu = d^2$. 
Coming from the Euler equation the parameters $\q$ and $\mu = d^2$ are discrete for fixed given $\bp$, 
labelled by the wave number $k$ in the $\xi$ direction in the mode ansatz for $\psi$.
Whenever convenient in the following we consider $(d, \q)$ as continuous parameters. 

In the linearised Euler equation the spectral parameter that determines instability is $\lambda$ respectively $c = i \lambda/k$
with wave number $k$, while the parameters of the equation are $\q  \mod 1$ and $d$ determined by $\bp$ and $k$.
In the Hill's equation \eqref{eq:hillgen} the spectral parameter is $\mu = d^2$, 
while the parameter of the equation is $c$, and the boundary conditions are determined by $\q$.
Ultimately what we are interested in is the relation between $c$, $d$, and $\q$, 
and this relation is determined by the spectral theory of the complex Hill's \cref{eq:hillgen}.

By multiplying \cref{eq:hillgen} by $\bar{g}$ and integrating, we recover Rayleigh's criterion (integrating by parts once): 
\begin{equation} \label{eq:rayleigh}
0 = \Re(\lambda)\int |g|^2 \frac{\psi'''}{|c-\psi'|^2} d\eta \,.
\end{equation}
Thus for instability where $\Im(c) \not = 0$ it is necessary that $\psi'''$ is positive somewhere and negative somewhere else, 
hence somewhere $\psi''' = 0$. Usually this is written in terms of velocity $U$,
decreasing the number of derivatives by one, and hence the name inflection point criterion.
In the periodic setting this criterion is not effective, since the derivative of a periodic function has zero average, and hence $U''$  
always changes sign.
This paper is not meant as an elaboration on this criterion, but instead we give a complete description of the point spectrum 
of the linearised Euler equations for the case that $\psi(\eta) = \cos(\eta)$.

\section{A complex Hill's Equation}\label{sec:hill}

\begin{figure}
\includegraphics[width=6cm]{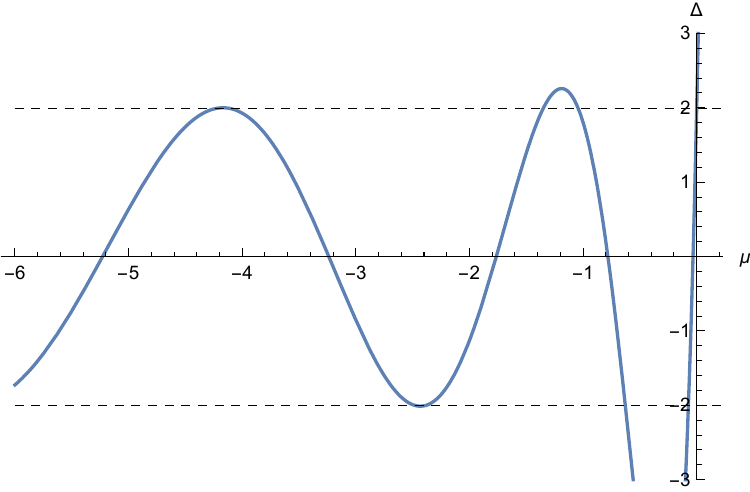} \hspace*{1ex}
\includegraphics[width=6cm]{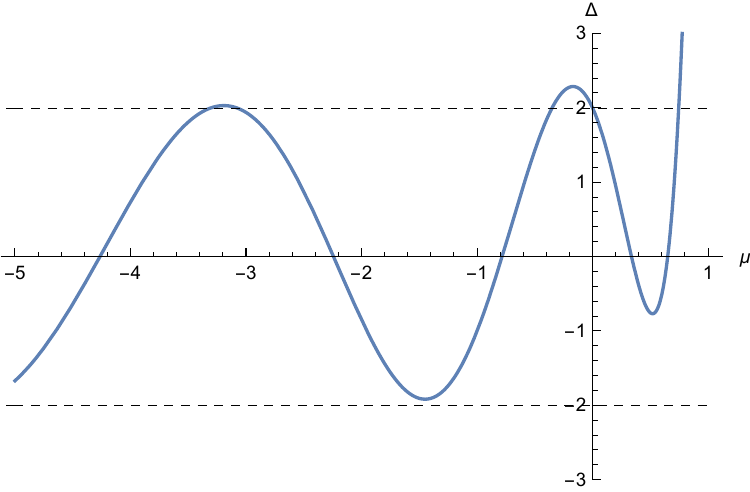}
\caption{Hill's discriminant $\Delta(\mu)$ for \cref{eq:hill}; dashed lines at $\pm 2$ indicate the range of $2 \cos 2 \pi \q$. 
Left: $c=2$; For real $c > 1$ there is no spectrum for $\mu \ge 0$.
Right: $c = 0.2 i$; For pureley imaginary $c$ there may be spectrum for $  \mu=d^2 > 0$ which is 
relevant for the Euler equation.
} \label{fig:delta}
\end{figure}

We now study  \cref{eq:hillgen} in detail for the specific example where $\psi(\eta) = \cos(\eta)$.
Substituting this into \cref{eq:hillgen} we obtain
\begin{equation}\label{eq:hill}
g'' + Q(\eta) g = \mu g, \quad Q(\eta) = \frac{\sin \eta}{c+\sin \eta} \,.
\end{equation}
Since $c \in \C$ this is a {\em Hill's equation} with complex valued periodic potential $Q(\eta)$ of period $2 \pi$. 
Much of the Hill's equation theory (for example from \cite{MW66}) is applicable, though a few key facts change. For example, the spectrum of the Hill's equation, that is, the values $\mu$ where there exists a bounded solution (on all of $\R$) to \cref{eq:hill}, is no longer real in general.

Writing Hill's equation \eqref{eq:hill} as a first order system the general solution is given by the principal fundamental solution 
matrix $M(\eta)$. After one period the matrix $M(2\pi)$ is called the monodromy matrix, and the trace of the monodromy 
matrix is called \emph{Hill's discriminant} $\Delta(\mu)$. Sometimes we will emphasise the dependence of 
the parameter $c$ that changes the function $Q(\eta)$ and will write $\Delta( \mu; c)$.
The boundaries of the intervals where bounded solutions exist are given by 
$\Delta(\mu) = 2$ with Floquet multipliers $+1$ and $\Delta(\mu) = -2$ with Floquet multipliers $-1$.
We are looking for solutions of \cref{eq:bound} with quasi-periodic boundary conditions with Floquet multipliers $e^{\pm 2\pi i \q}$, 
such that we require 
\begin{equation} \label{eq:disdef}
     \Delta(\mu) = 2 \cos( 2 \pi \q) \,.
\end{equation}
In \cref{fig:delta} we show the function $\Delta(\mu)$ for real and purely imaginary values of $c$.

We are now going to discuss some properties of the spectrum of the complex Hill's equation \eqref{eq:hill}. 
First we recall the classical case with real $c > 1$ (see \cref{fig:delta}). 
Then we consider the case of purely imaginary $c$ (see figures~\ref{fig:deltacimag} and~\ref{fig:cdplot}),
and finally the case of complex $c$ (see \cref{fig:complexmu}).

\subsection{Real $c$}
When $c \in \R$ then $Q(\eta)$ is a real valued $2\pi$-periodic function and the standard theory of the
Hill's equation applies when $c > 1$ (or $c < -1$) such that $Q(\eta)$ is smooth. \footnote{It is usually formulated for $\pi$-periodic function with average zero; in our case the average $g_0$ is non-zero and shifts the spectral parameter $\mu$.}
In particular we have (see, e.g.,~\cite{MW66,DLMF})
\begin{theorem}
For given $\q$ the equation $\Delta(\mu) = 2 \cos 2\pi \q$ has infinitely many solutions.
There is a largest solution $\mu^*$ such that for $\mu > \mu^*$ there are no solutions.
For analytic $Q(\eta)$ the discriminant $\Delta(\mu)$ is an entire function of $\mu$.
\end{theorem}
In the standard theory the so called region of stability is defined as those $\mu$ for which $|\Delta(\mu)| \le 2$.
Even though we are interested in the case of positive $\mu = d^2 = k^2/p^4$ for the Euler equation, 
in \cref{fig:delta} we include negative $\mu$ to connect to the classical theory of Hill's equation. 

\subsection{Purely imaginary $c$}\label{subsec:pureimag}

\begin{figure}
\includegraphics[width=4cm]{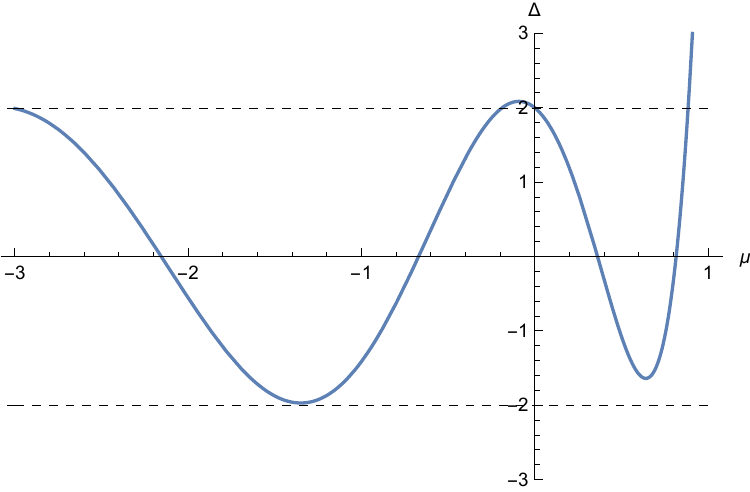} \hspace*{1ex}
\includegraphics[width=4cm]{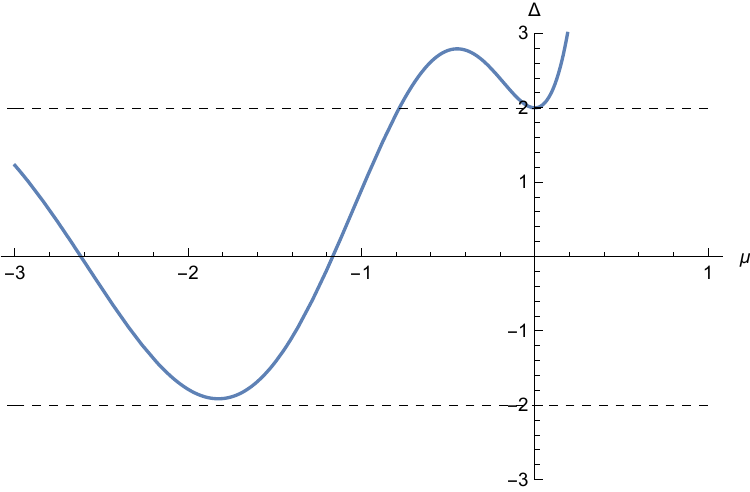} \hspace*{1ex}
\includegraphics[width=4cm]{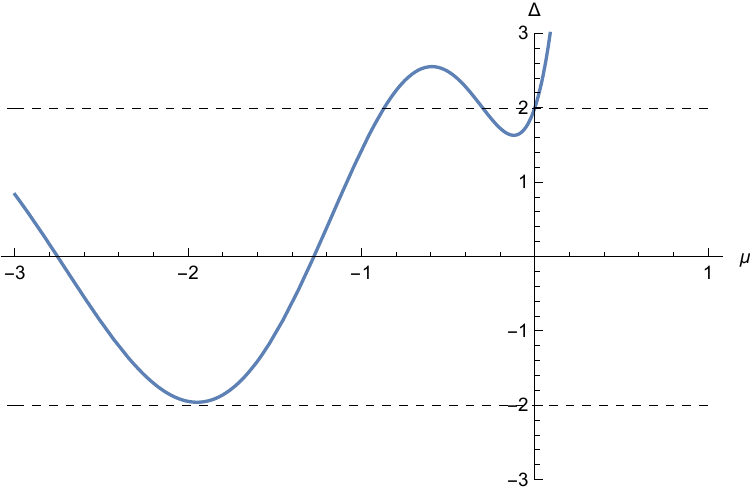}
\caption{Hill discriminant $\Delta(\mu)$ for \cref{eq:hill} with purely imaginary $c$ from left to right $c =  i/10, i/\sqrt{2}, i$;
dashed lines at $\pm 2$ indicate the range of $2 \cos 2 \pi \q$. 
For small $|c|$ there is spectrum with $ 0 < \mu < 1$. 
When $|c|$ reaches $1/\sqrt{2}$ the spectrum disappears and $\Delta'(0) = 0$.
For $|c| > 1/\sqrt{2}$ there is no spectrum with $\mu > 0$.
} \label{fig:deltacimag}
\end{figure}

As is evident in  \cref{fig:delta} right, the discriminant and the spectrum of the Hill's equation is real even though $Q(\eta)$ is 
not real for purely imaginary $c$.
The potential $Q(\eta)$ has the property that
$    
Q(\eta) = \overline{ Q(-\eta) }
$
which is called space-time or PT-symmetry \cite{Bender07},
which is a special case of so-called pseudo-hermiticity \cite{Scholtz92}, 
which means that there is a unitary involution that intertwines the operator with its adjoint \cite{caliceti04}.
The special case of the PT-symmetric Hill's equation has been studied in \cite{Bender99}.
The main conclusion from PT-symmetry in our case is that for purely imaginary $c$ Hill's discriminant is real 
even though the potential $Q(\eta)$ is complex, which we will see in the next section.
Note that as in \cite{Shin04} this does not imply that the whole spectrum is real. Indeed, in \cref{fig:delta} right, the local minimum of $\Delta(\mu)$ strictly between the lines $\Delta = \pm 2$ indicates the presence of eigenvalues with both nonzero real and imaginary parts. 

When $c$ is purely imaginary and sufficiently small there is some spectrum for $\mu > 0$, see, for example, \cref{fig:delta} right. 
The change in $\Delta(\mu)$ with increasing purely imaginary $c$ is illustrated in \cref{fig:deltacimag}.
From these graphs we observe that for $c = i\beta$ there is some spectrum with $\mu > 0$ near $\mu = 0$
for $0 \le \beta \le 1/\sqrt{2}$ and none for $\beta > 1/\sqrt{2}$. 
In the boundary case (\cref{fig:deltacimag} middle) we have $\Delta'(0) = 0$. 
The condition for existence of some spectrum with $\mu > 0$ is that $\Delta'(0) < 0$, since asymptotically 
$\Delta( \mu) \to +\infty$ for $\mu \to +\infty$.

\begin{figure}
\includegraphics[width=7cm]{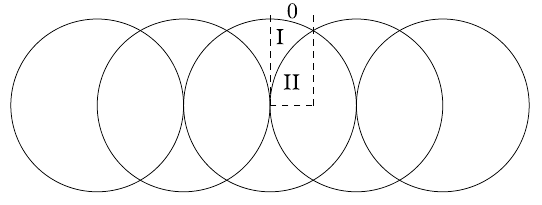}
\caption{Unit circles in the $(\q, d)$-plane with centres at $( l, 0)$ for  $l=-2,\dots,2$. 
On these circles $c = 0$ is in the spectrum. 
A fundamental domain is indicated by dashed lines. The numerals indicate the number of circles in which a given point lies, see \cref{def:regions}.}
\label{fig:circles5}
\end{figure}

The limiting case $\beta \to 0$ along the imaginary axis (and hence $c = 0$ in the limit) is easily solvable because then $Q \equiv 1$.
This limit corresponds to a case when the parameters of Euler's equation are on the boundary of the unstable disk, compare \cref{thm:disk}.
\begin{lemma}\label{lem:circle}
In the case $c = 0$ the real spectrum $\mu = d^2$ of the Hill's equation \eqref{eq:hill}  with boundary conditions \cref{eq:bound}
satisfies $(\q + l)^2 + d^2 = 1$ for any integer $l$.
\end{lemma}
\begin{proof}
When $c = 0$ we have $Q = 1$, and the Hill's equation \eqref{eq:hill} simplifies to
$$
g'' + (1-d^2)g = 0 
$$
with solutions $g(\eta) = \exp( \pm i \sqrt{ 1 - d^2} \eta)$.
The  boundary condition \cref{eq:bound}  is
$g(2 \pi) = \exp(2\pi  i \q) g(0)$ and
gives $\q+l = \pm \sqrt{1-d^2}$ for integer $l$. 
Squaring both sides and re-arranging gives the result. 
\end{proof}

The equation $(\q + l)^2 + d^2 = 1$ for integer $l$ defines a sequences of circles  of radius one in the $(\q, d)$-plane.
The centres of the circles are at $(-l, 0)$. Neighbouring circles intersect in two points, and circles whose $l$'s differ by 2 touch in one point,
see \cref{fig:circles5}. Considering $\q$ as an angle with period 1 this becomes a picture
of a single circle with diameter two wrapped around a cylinder with circumference 1.
A convenient fundamental domain is $0  \le \q \le 1/2$ and $d \ge 0$ as indicated in \cref{fig:circles5}.
Since Floquet multipliers come in pairs $\exp( \pm 2 \pi i \q)$ one can restrict to non-negative $\q$.
Since the spectral parameter $\mu = d^2$ we can restrict to $d \ge 0$.

\begin{figure}
\includegraphics[width=8cm]{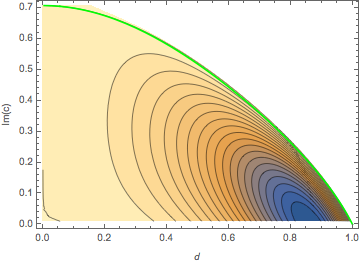}
\caption{Contours of $\Delta(\mu=d^2, c)$ on the  plane $(d, \Im(c))$ for $\Re(c) = 0$. The contours are equidistant in values of $\Delta$ from $-2$ to $2$ in steps of $0.2$. The maximum of $\Delta = 2$ occurs at $c = i/\sqrt{2}$, $d = 0$ and the curve
extends to $d = 1$. The minimal value $\Delta = -2$ is attained only at the single point 
where $c = 0$ and $d = \sqrt{3}/2$. The green curve is an approximation obtained from 
a $3\times 3$ Hill determinant $D_1$, see \cref{eq:hilldet33}. 
}
\label{fig:cdplot}
\end{figure}

Combining graphs like those in \cref{fig:deltacimag} into a single diagram we obtain \cref{fig:cdplot}, where we use $d$ instead of $\mu = d^2$ 
on the horizontal axis. The regions of stability where $|\Delta(d^2; c)| < 2$ are shaded and contourlines indicate constant $\Delta$ 
and the constant boundary condition given by $\q$. 
Note that his diagram directly relates the spectral parameter $\mu = d^2$ of the Hill's equation 
to the spectral parameter $c$ of Euler's equation for the case of purely imaginary $c$.
Along the axis $c = 0$   \cref{lem:circle} shows that $\Delta(d^2,c=0) = 2 \cos( 2 \pi \sqrt{ 1 - d^2})$.
Hence the minimum $\Delta = -2$ occurs at $d = \sqrt{3}/2$ corresponding to the intersection of the circles
in \cref{fig:circles}.

The intersection of the contour $\Delta = 2$ with the boundary $\mu =d= 0$ occurs 
at $c = i/\sqrt{2}$, as can also be seen in \cref{fig:deltacimag} (middle), where $\Delta'(\mu)=0$ at $\mu = 0$.
The following lemma proves this claim.
\begin{lemma}
The derivative of $\Delta$ with respect to $\mu$ at $\mu = 0$ is
\begin{equation} \label{eq:delmu}
     \Delta'(0) = \frac{2\pi^2c(1+2c^2)}{(c^2-1)^{3/2}}. 
\end{equation}
\end{lemma}
\begin{proof}
 \Cref{eq:hill} is of the form $g'' - h'' h^{-1} g = \mu g$ with 
$ h(\eta) = c + \sin\eta$, or more generally $h(\eta) = c -\psi'$ from \cref{eq:hillgen} where $\psi$ is a periodic stream function. 
Thus $g_1 = h$ is a solution for $\mu=0$ and $c \notin [-1,1]$.
Reduction of order gives a second solution $g_2 = h \int  h^{-2} \, d\eta$. 
Following \cite{MW66} the derivative $\Delta'$ is found by differentiating \cref{eq:hill} with 
respect to $\mu$ where $g = g(\eta; \mu)$. The resulting ODE is again \cref{eq:hill} but with a non-homogeneous term
and can be solved using variations of constants. 
{\color{black} 
Defining the principal fundamental solution matrix $$\begin{pmatrix} y_1(\eta;0) & y_2(\eta;0)\\ y_1'(\eta;0) & y_2'(\eta;0) \end{pmatrix} := \begin{pmatrix} g_1(\eta;0) &g_2(\eta;0) \\ g_1'(\eta;0) & g_2'(\eta;0) \end{pmatrix} \begin{pmatrix} g_1(0;0) &g_2(0;0) \\ g_1'(0;0) & g_2'(0;0) \end{pmatrix} ^{-1},$$
we apply formula (2.13) from \cite{MW66}. Combining all of the terms into one integral, and using the facts that $g_2(0) = 0$ and that $h$ and $h'$ are both periodic, formula (2.13) reduces to 
\begin{equation} \label{eq:deltaprime}
\Delta'(0) = - \int_0^{2\pi} h^2 d\eta \cdot \int_0^{2\pi} h^{-2} d \eta.
\end{equation}
}

The first integrals is elementary, and the second integral can be found using the method of residues,
and hence the result follows.
\end{proof}

{\color{black}
\begin{remark}
We note here that \cref{eq:deltaprime} is somewhat general. All that is required for it to hold is the existence of a non-vanishing periodic solution to the Hill's equation and formula (2.13) will reduce to \cref{eq:deltaprime}.
\end{remark}
}

The lemma shows that $\Delta'(0)$ is indeed zero at $c =i/\sqrt{2}$, at the top left corner in \cref{fig:cdplot}.
For $c\notin [-1,1]$, either real or purely imaginary, assuming that $\Delta'(0)$ is real and negative
(as shown in \cref{fig:delta} right and \cref{fig:deltacimag} left), there are points of spectrum of the Hill's equation nearby with real and positive $\mu$. 
These  satisfy some quasi-periodic boundary conditions, and hence there is some spectrum of \cref{eq:hill} on the real axis. 
For imaginary $c = i \beta, \beta\geq0$, the derivative $\Delta'(0)$ is negative for $\beta \in (0, 1/\sqrt{2})$ according to the lemma,
and hence there is spectrum nearby, as can be seen in \cref{fig:cdplot} to the right of the $d=0$ axis.

\subsection{Complex $c$}

For complex $c$ the Hill's equation \eqref{eq:hill} generally has a complex monodromy matrix, and hence also $\Delta(\mu)$ 
is complex in general. Our boundary conditions do imply, however, that $\Delta(\mu) = 2 \cos 2 \pi \q$ needs to be real.
This is a complex equation, so that we require $\Im(\Delta(\mu)) = 0$, which is indicated by red lines in \cref{fig:complexmu}.
Even though  complex $\mu$ is irrelevant for the application to the Euler equation it is helpful to see how 
the spectrum with real positive $\mu$ (if it exists) extends to the complex $\mu$-plane.
In \cref{fig:complexmu} we show a contourplot of $\Re(\Delta(\mu))$ on the complex $\mu$-plane for fixed $c$.
Overlaid are the curves where $\Im(\Delta(\mu))=0$ in red.
The reason to consider the curves of spectrum in the complex domain is that what we are looking for 
the special case when these curves intersect  the axis of real $\mu$ for positive $\mu$ and with $|\Delta(\mu) | < 2$.
In \cref{fig:complexmu} middle such a case can be seen as the intersection of the red curve of spectrum with the positive 
real $\mu$ axis inside the shaded region where $|\Delta(\mu) | < 2$ such that the boundary conditions are satisfied.
This intersection means that the complex value of $c$ for which this figure is drawn is a complex eigenvalue 
of the Euler equation with wave number $k=\pm \sqrt{\mu} p^2$.

\begin{figure}
\includegraphics[width=4.8cm]{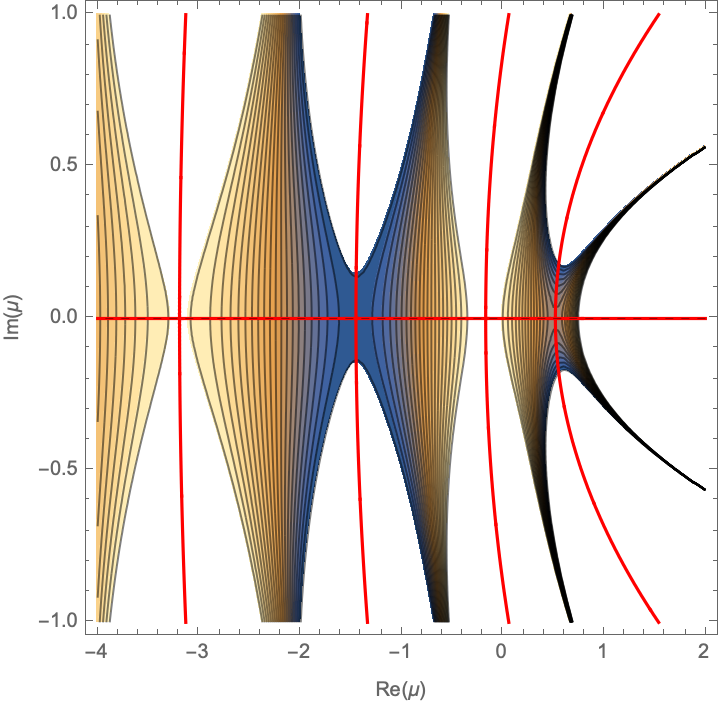}
\includegraphics[width=4.8cm]{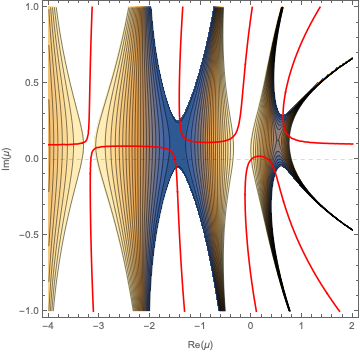}
\includegraphics[width=4.8cm]{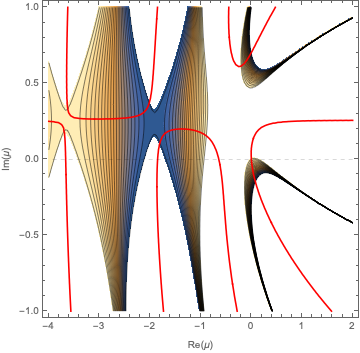}
\caption{Contours of $\Re(\Delta(\mu)) $ on the complex $\mu$-plane with zero-locus of $\Im(\Delta(\mu))$ in red
for $c = i 0.2$ (left), $c = 0.1 + i 0.2$ (middle) and $c=0.5 + i 0.7$ (right). 
Contourlines are corresponding to values $-2, \dots, 2$ (from yellow to blue) in steps of $0.2$.
In the middle figure there is a real {\em positive} value of $\mu$ in the spectrum, but not in the right figure.
}
 \label{fig:complexmu}
\end{figure}

After this brief overview of the spectrum of the complex Hill's equation we are now going to compute the
Hill discriminant as a Fredholm determinant and using this formula construct an Evans function for the linearised Euler equation.

\section{A Hill determinant}\label{sec:det}

In order to prove facts about the spectrum of the linearised Euler equations via the discriminant of Hill's equation we use Hill's determinant. 
This is slightly more general than in \cite{MW66} because here $Q(\eta)$ 
is a complex valued function, but this does not affect the computation of the Hill determinant.
We need to first find the complex Fourier series of the function 
$Q(\eta)$ defined in \cref{eq:hill}.
The Fourier series is  conveniently expressed in terms of a parameter $s$ related to $c$ by
\begin{equation}  \label{eq:s}
    c = \frac{1}{2}\left( s+ \frac{1}{s} \right) \,.
\end{equation}
This is the well-known conformal Joukowski map that transforms the upper half-disk of radius 1
into the lower half plane. The upper half of the unit circle $s = e^{i \phi}$ is mapped onto the unit interval 
$c = \cos \phi$, and circles of radius $r < 1$ are mapped to ellipses with semi-major axis $(r+1/r)/2$ 
and semi-minor axis $(r-1/r)/2$.
The two solutions of the quadratic equation for $s$ in terms of $c$ are combined to produce a single-valued
inverse map from $s(c)$ such that for $|c| \to \infty$ we have $s \to 0$ in the whole complex plane. 
The mapping $s(c)$ is not continuous along the branch cut where the argument of the square root is negative,
which occurs along the real interval $[-1, 1]$. Later we will see that this interval corresponds to the continuous 
spectrum of the linearised Euler equation.
The mapping $s(c)$ is holomorphic on $\C \setminus [-1, 1]$.
One can check that this mapping from $c$ to $s$ implies $|s| \le 1$ everywhere 
and $|s| = 1$ is attained along the branch cut. 
The inverse function $s(c)$ satisfies $s(c) = -s(-c)$ and $s(\bar c) = \overline{s(c)}$, and 
so it is sufficient to consider the positive quadrant of the $s$-plane.
Substituting $c = c(s)$ into $Q(\eta)$ gives 
\begin{equation}
Q(\eta)   = \frac{1}{\frac{i \left(s+\frac{1}{s}\right)}{\left(e^{i\eta}-e^{-i\eta}\right)} + 1}  = 
1 - \frac{1-s^2}{1+s^2}\left( \frac{1}{1-  s i e^{i\eta}} + 1 - \frac{1}{1 -  s /( i e^{i\eta}) }  \right)  
\end{equation}
Expanding this as geometric series in $s$ gives the Fourier series $Q(\eta) = \sum_{k=-\infty}^{\infty} g_k e^{ik \eta}$ where 
$$
 g_0 = 1 + \kappa, \quad 
g_k = \kappa i^k s^{|k|} , 
\quad
\kappa  = -\frac{ 1 + s^2}{ 1 - s^2} \,.
$$
Because $|s|<1$ the Fourier coefficients decay exponentially and the series converges.
Notice that because $|s| < 1$ the factor $\kappa$ is in the left half plane, $\Re(\kappa) < 0$.
Expanding a solution of Hill's equation as a Fourier series the solvability condition leads to the normalised Hill determinant.
Following \cite{MW66} in our particular case we find
\begin{proposition}
The Hill-determinant of Hill's equation \eqref{eq:hill} is
\begin{equation} \label{eq:HillD}
    D(\Lambda)= \left|\left| \frac{\tilde{g}_{n-m}}{\Lambda-n^2} + \delta_{n,m} \right|\right|
\end{equation}
where $\tilde{g}_{k\neq0} = g_{k}$ and $\tilde{g}_0 =0$. The Hill-discriminant is
\begin{equation}\label{eq:discrim}
            \Delta(\mu) =  
                  2 - 4 D(g_0-\mu) \sin^2(\pi \sqrt{g_0-\mu}) \,.
\end{equation}
\end{proposition}

This is Hill's fundamental result which relates the trace of the principal fundamental solution matrix of \cref{eq:hill}
-- the Hill discriminant $\Delta(\mu)$ -- to the Hill determinant $D$.
In his original work \cite{Hill1878} Hill used  these equations to determine the Floquet exponent and hence the stability of a particular solution.
More generally Hill's formula can be viewed as an equation for the spectral parameter $\mu = d^2$ for which quasiperiodic solutions with 
Floquet exponent $\q$ exist. 
Instead in our setting the spectral parameter $d$ and the Floquet exponent $\q$ are given, 
and we consider this as an equation for the parameter $c$ (or, equivalently $s$)
for which such a solution exists. Here the Fourier coefficient $g_0$ and the higher order coefficients $g_k$ hidden in the 
definition of Hill's determinant all depend on $s$.
This is a change of interpretation by which the original spectral parameter $\lambda$ and hence $c$ of the 
linearised Euler equations becomes a parameter in Hill's equation, while the spectral parameter $\mu = d^2$ in Hill's equation
is a parameter in the separated Euler's equation.

The three cases discussed in the previous section now appear like this:
1)~For real $c > 1$ we have real $s < 1$ and hence the Hill determinant is Hermitian and the Hill discriminant is real for real $\mu$.
2)~For purely imaginary $c$ we have purely imaginary $s$, so that all $g_k$ are real because $i^k s^{|k|}$ gives an even power of $i$.
Hence the Hill determinant is real and the Hill discriminant is real for real $\mu$.
3)~For complex $c$ also $s$ is complex and the Hill determinant and discriminant are complex.

The beauty of Hill's formula is that it gives a very good approximation already for small truncation sizes. 
Instead of the determinant of a bi-infinite matrix $D(\Lambda)$ denote by $D_k(\Lambda)$ the determinant
of the finite matrix truncated to $m,n = -k, \dots, k$. Thus the truncation of $D(\Lambda)$ with $k=1$ (and hence matrix
size $3\times 3$) becomes
\begin{equation}
 D_1(\Lambda)  = \left|\left|
\begin{array}{ccc}
 1 & -\frac{i \kappa s}{\Lambda } & -\frac{\kappa s^2}{\Lambda -1} \\
 \frac{i \kappa s}{\Lambda -1} & 1 & -\frac{i \kappa s}{\Lambda -1} \\
 -\frac{\kappa s^2}{\Lambda -1} & \frac{i \kappa s}{\Lambda } & 1 \\
\end{array}
\right|\right| \,.
\end{equation}
The quality of approximation was important in Hill's original application to the Moon, where he found good results
already from $3\times 3$ determinants \cite{Hill1878}. Even for such severe truncation a figure similar to \cref{fig:cdplot}
computed from the approximation $D_1(\Lambda)$ gives a qualitatively correct answer.
This is particularly true when asking for $\Delta = 2$ which implies $D = 0$.
Now $D$ factors because of discrete symmetry and expressing $s^2$ in terms of $\kappa$ 
in  the $3\times 3$ truncation of the Hill determinant gives
\begin{equation} \label{eq:hilldet33}
   D_1( 1 + \kappa - d^2) = \frac{ d^2 (  ( 2 + d^2) \kappa - d^2 ) ( 3 \kappa^2 - d^2 \kappa - 1 + d^2)}{ (1 + \kappa - d^2)( d^2 - \kappa)^2(1 - \kappa)^2}
\end{equation}
The vanishing of the last factor in the numerator
is shown as a green line in \cref{fig:cdplot}. The two endpoints of the green line correspond to the roots of this factor where
$(d, \kappa) = (1, 0)$ and hence $c =0$ and $(d, \kappa) = (0, -1/\sqrt{3})$ and hence 
$c = i/\sqrt{2}$. It should be noted that even though the result at the two endpoints is exact, 
for intermediate points along the boundary curve there are corrections from incorporating higher order
approximations to the determinant.
As already noted, along the lower boundary of \cref{fig:cdplot} where $c = 0$ the exact answer
$\Delta(d^2)$ is obtained from Hill's formula where $D = 1$ since $s = \pm i$ and hence $\kappa = 0$.
When $d=0$ we always have $g = c + \sin\eta$ as a solution that is $2\pi$ periodic, and hence
$\Delta = 2$. For the Hill discriminant this leads to $\Delta(0) = 2$, and hence $D(g_0) = 0$. 
This surprising identity appears because for $d = 0$ the three central columns of the Hill determinant 
are linearly dependent.
It should be noted that the $3\times 3$ approximation to the Hill determinant produces the 
exact values at the edges $d =0$ and $c = 0$ of \cref{fig:cdplot}.

In \cref{sec:evans} we show that when the circles in $(\q,d)$ space from \cref{lem:circle} where $c=0$ are crossed in the correct
direction then a pair of real eigenvalues is born. This is based on the following well-known formula, see, for example, \cite{GohbergKrein69}:
\begin{theorem} \label{lem:freddet}
The derivative of a Fredholm determinant $f = \det F$ is given by $\dot f = f {\rm tr} (F^{-1} \dot F).$
\end{theorem}

In general the computation of $F^{-1}$ may of course be difficult, but in our case this turns out to be simple at $c = 0$.

\begin{lemma} \label{lem:trace}
The derivative of the Hill determinant $D(g_0 - \mu)$ with respect to $\mu$  at $c = 0$ vanishes.
The derivative of the Hill determinant $D(g_0 - \mu)$ with respect to $c$ at $c = 0$ vanishes.
In both cases it is assumed that $1 - \mu \not = n^2$ for any integer $n$.
\end{lemma}
\begin{proof}
From \cite{MW66} we have that $D(g_0-d^2) = \det F$ is a Fredholm determinant.
Indeed, there it is shown 
to be of Hill type, 
which means that the sum of the absolute values of the off-diagonal entries converges.
A necessary condition for this is $\sum |g_k| < \infty$ 
which holds for $|s| < 1$ and for $s = \pm i$ (so that $\kappa = 0$), 
and hence for all $c \in \C$ except for $c \in [-1,0) \cup (0, 1]$.
The assumption that $1 -\mu \not= n^2$ ensures that the denominators $\Lambda - n^2$ in \cref{eq:HillD} are non-zero.
So in particular we can apply \cref{lem:freddet} at $c = 0$ 
{\color{black}
as long as $d \not = 0$ and $d \not = 1$.
}
For this we need $F^{-1}$ at $c =0$. Now $c = 0$ implies $s = \pm i$ and $\kappa = 0$, 
hence all Fourier coefficients $g_k = 0$ and $F$ and $F^{-1}$ are the identity.
Thus the trace of $\dot F$ needs to be computed. 
For either derivative the main observation is that the diagonal of $F$ is constant, 
and hence the diagonal of $\dot F$ vanishes.
Thus ${\rm tr} \dot F = 0$ and both results follow.
\end{proof}

Note that this proof was easy because for $c =0$ we have $Q = \psi'''/\psi' = \pm 1 = const$ for $\psi'(\eta) = \sin(\eta)$ or $\cos(\eta)$, 
and hence the Fourier series of $Q$ is trivial.
For a more general periodic function $F$ has additional entries, and hence $F^{-1}$ is more complicated,
and even though the diagonal of $\dot F$ still vanishes the result does not immediately follow.

{\color{black}
Even though $d=1$ is excluded in the above proof, the Hill discriminant (and its derivatives) are of course well defined
in this case as well. This is so because the Hill-determinant $D$ is multiplied by $\sin^2( \pi \Lambda)$ in 
\cref{eq:discrim}, and $\Lambda = 0$ for $d = 1$ and $c = 0$.
}

\section{An Evans function}\label{sec:evans}

\begin{figure}
\includegraphics[width=4.4cm]{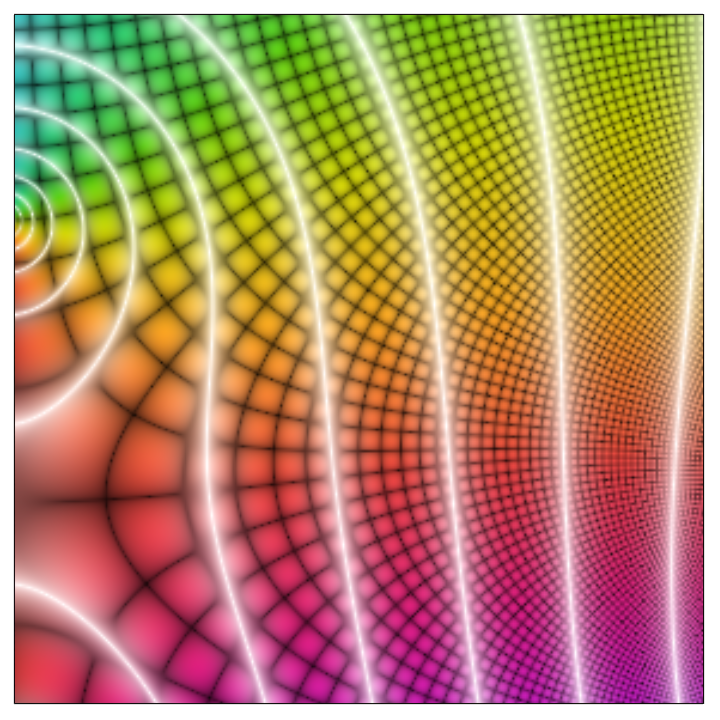}
\includegraphics[width=4.4cm]{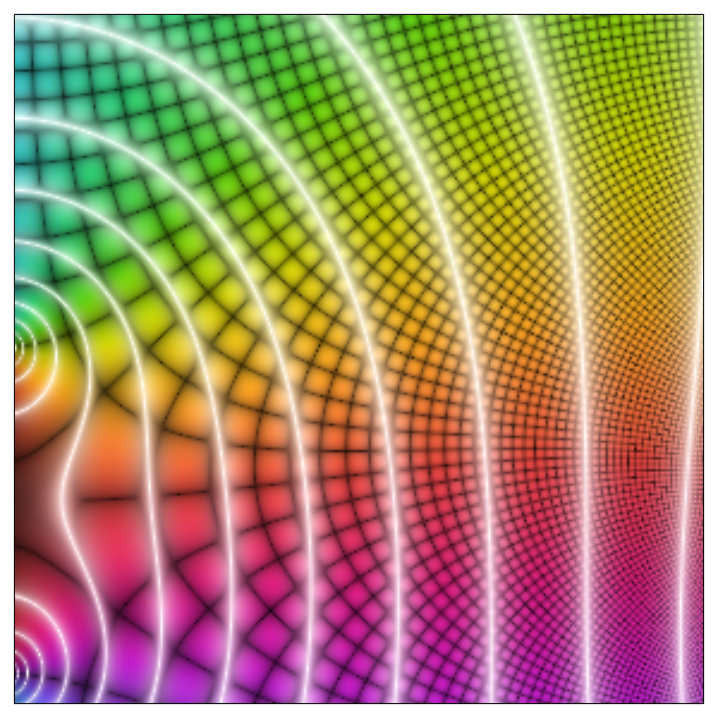}
\includegraphics[width=4.4cm]{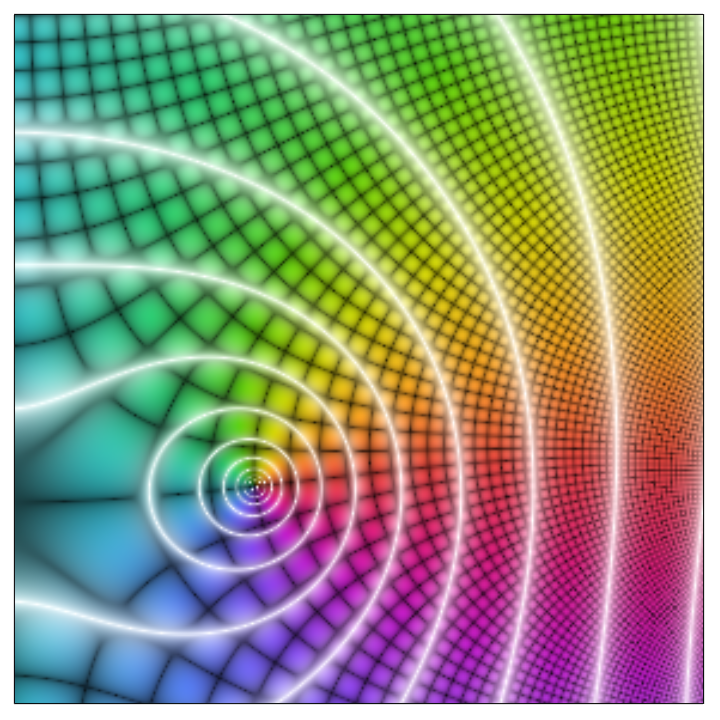}
\caption{Complex colour wheel graphs of the Evans function $E(c)$ for $d=0.6$, $\q = 0.1$, $0.22$, $0.4$, showing the positive 
octant of the complex $c$ plane up to $(1 + i)0.6$. Shown are cases with a pair of imaginary roots, two pairs of imaginary roots, 
and a quadruplet of complex roots, respectively.}
\label{fig:evans}
\end{figure}

Hill's determinant allows us to compute Hill's discriminant, which is the trace of the fundamental 
solution matrix of \cref{eq:hill} after one period. In our setting this trace is given by $2 \cos 2\pi \q$. Unlike the Hill discriminant which foremost is a function of the spectral parameter $\mu$ of Hill's equation, 
we define the Evans function as a function of the spectral parameter $c$ of the linearised Euler's equation.
Using Hill's discriminant we can build an analytic function away from $c \in [-1,1]$ that vanishes when $c$ %
is in the discrete spectrum of the linearised Euler operator. Moreover, the degree of vanishing (in $c$) is the same as the algebraic multiplicity of 0 in the spectrum of \cref{eq:sep}, or equivalently, of $\mu$ in \cref{eq:hillgen} with boundary conditions \eqref{eq:bound}. For this reason, we call it an Evans function for \cref{eq:sep}, and (in a slight abuse of language) for the `class' operator associated with it, which is one block of the block-diagonalised form from \cite{Li00} of the linearised Euler equations. 
\begin{theorem} 
The function
\begin{equation}   \label{eq:evans} 
    E(c; \q, d) = -2 + 2 \cos ( 2 \pi \q) + 4 D(g_0-d^2) \sin^2 \pi \sqrt{g_0-d^2}
\end{equation}
is an Evans function for the separated ODE \eqref{eq:sep} 
where $D(g_0-d^2)$ is the Hill determinant depending on $c$ through the Fourier coefficients $g_k$ of $Q(\eta) = \sin \eta/(c+\sin\eta)$. That is, $E(c;\theta,d)$ is analytic in $c$ outside of $[-1,1]$ and vanishes for values of $c$ where $0$ is in the spectrum of \cref{eq:sep} with $f$ having periodic boundary conditions. The degree of vanishing of $E$ as a function of $c$ is the same as the algebraic multiplicity of $0$ as an eigenvalue of \cref{eq:sep} or, equivalently of $d$ as an eigenvalue of \cref{eq:hillgen} with boundary conditions \eqref{eq:bound}. 
\end{theorem}

\begin{figure}
\includegraphics[width=6cm]{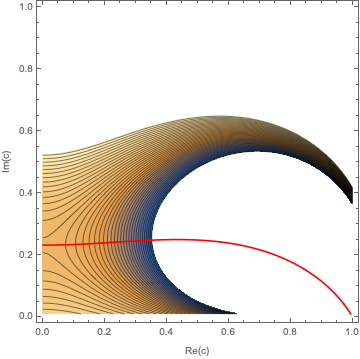}
\includegraphics[width=6cm]{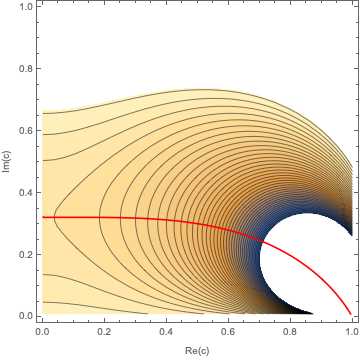}
\caption{Contours of $\Re(\Delta(\mu;c))$ on the complex $c$-plane with zero-locus of $\Im(\Delta(d^2;c))$ in red
for $d= 1/2$ (left) and $d = 1/4$ (right).
Contourlines are corresponding to values of $2\cos(2\pi\q)$ ranging from $-2, \dots, 2$ (from yellow to blue) in steps of $0.1$.
The red curve indicates the location of the spectrum of Euler's equation for $\mu = d^2$ and $\q$ determined by the contourline intersected.} \label{fig:complexc}
\end{figure}

\begin{proof}
Equating equations \eqref{eq:disdef} and \eqref{eq:discrim} we see that by construction
we have that $E(c) = 0$ when $c\not \in [-1,1]$ is in the discrete spectrum of the Euler equation. The rest of the proof is an application of Lemma 8.4.2 from \cite{KP13}. In our notation, Lemma 8.4.2 from \cite{KP13} says that the function 
$$
2 \cos(2\pi \q) - \Delta(\mu) 
$$
is entire in $\mu$ for fixed $\q$ and it is entire in $\q$ for fixed $\mu$. 
Moreover, according to the lemma, the algebraic multiplicity of $\mu$ as an eigenvalue of \cref{eq:hillgen} with boundary conditions \eqref{eq:bound} is equal to the degree of vanishing of $E$ (as a function of $\mu$). 
Differentiating  \cref{eq:discrim} shows that 
\begin{equation}
\frac{\partial \Delta}{\partial c} = - \Delta'(\mu) g'_0(c),
\end{equation}
whence 
$$
\frac{\partial \Delta}{\partial c} = 0 \Leftrightarrow \Delta'(\mu) = 0, 
$$
because the other factor is non-vanishing, 
$$|g_0'(c)| = \frac{1}{(c^2 -1)^{3/2}}.$$ 
Now suppose that $\Delta'(\mu) = 0$. Differentiating again, we have that 
$$ 
\frac{\partial^2 \Delta}{\partial c^2} = - \Delta''(\mu) (g'_0(c))^2, 
$$ 
and again 
$$ \frac{\partial^2 \Delta}{\partial c^2} = 0 \Leftrightarrow \Delta''(\mu) = 0.$$
Repeating the procedure shows that the degree of vanishing of $\Delta$ as a function of $c$ (and hence the degree of vanishing of $E$ on an isolated point of spectrum of the linearised Euler equations) is the same as the degree of $\Delta$ as a function of $\mu$, which by Lemma 8.4.2 from \cite{KP13} is the same as the algebraic multiplicity of \cref{eq:hillgen} with boundary conditions \cref{eq:bound}, or equivalently, the algebraic multiplicity of $0$ as a periodic eigenvalue of \cref{eq:sep}. 
\end{proof}

For the Evans function $\mu = d^2$ and $\q$ are parameters, and we are trying to find a (complex) $c$ 
such that  $\Delta(d^2;c) = 2 \cos 2 \pi \q$.
In \cref{fig:complexc} we show a contourplot of $\Re( \Delta(d^2;c))$ on the complex $c$-plane for fixed $d^2$.
Overlaid are the curves where $\Im( \Delta(\mu;c)) = 0$ in red. Since the dependence of the Evans function on 
$\q$ only appears in the single cosine term, this presentation gives the location of complex roots of 
the Evans function for fixed $d^2$ and varying $\q$.

We note that we recover the circles from \cref{lem:circle} from $E(0; \q, d) = 0$:
When $c=0$ then $s=\pm i$ and hence $g_0 = 1$ and $\tilde g_k =0$ which implies  $D(1-d^2) = 1$, 
such that  %
\[
E(0; \q, d) = -4  \sin^2( \pi \q) + 4 \sin^2( \pi\sqrt{ 1- d^2})
\]
and thus on the circles $( \q + l)^2 + d^2 = 1$ the Evans function vanishes. A contour-plot of $E(0; \q, d)$ is shown 
in \cref{fig:eva} left. At the minimum $(\pm 1, 0)$ the value is $-4$, while at the maximum $(0, \sqrt{3}/2 )$ the value is 4.
The critical point at the origin is degenerate. 
In the right part of \cref{fig:eva} the eigenvalue is fixed to $c = 0.1 i$, and the 0-contour shows for which values 
of $(\q, d)$ this eigenvalue occurs. Since the dependence of the Evans function on $\q$ is simple the zero 
contour can be given explicitly by solving $E = 0$ for $\q$ such that
\[
      \sin^2 \pi\q =  D( g_0 - d^2)  \sin^2 \pi \sqrt{ g_0 - d^2}
\]
where the right hand side is considered as a function of $d$ for constant $c$. 
The maximum of the 0-contour of $E=0$ occurs where $D(g_0 - d) = 0$. 
All 0-contours start at the origin because $D(g_0) = 0$. 

\begin{figure}
\includegraphics[width=5cm]{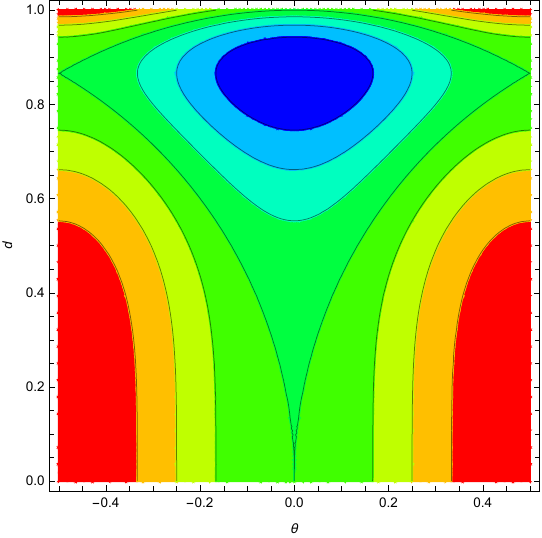}
\includegraphics[width=5cm]{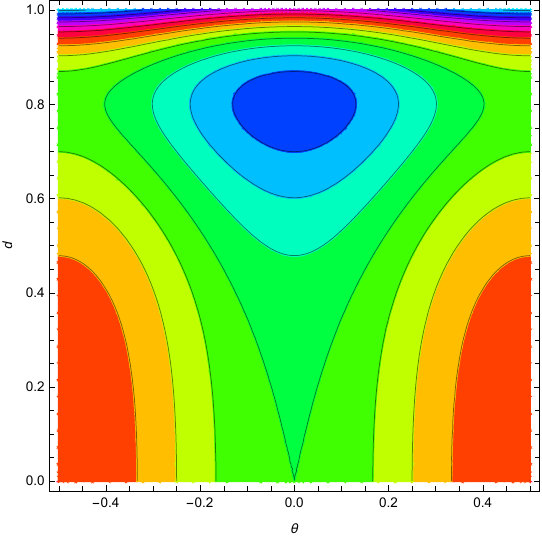}
\caption{Contourplot of $E(c=0; \q, d)$ (left)  and of $E(c=0.1 i; \q, d)$ (right) as
a function of the parameters $\q$ and $d$.
Contour lines have integer values, and in the green shading is the 0-level that shows
where  the corresponding eigenvalue $c$ occurs (left $c=0$, right $c = 0.1 i$).}
\label{fig:eva}
\end{figure}

One property of an Evans function is that it limits to a constant for large arguments:
\begin{lemma} \label{lem:infinity}
The function $E(c; \q, d)$ decays to a non-positive real constant for large $|c|$ 
$$
\lim_{|c| \to \infty} E(c; \q, d) =  2\cos(2\pi \q) - 2\cosh(2\pi d) = - 4 \sin^2( \pi \q) - 4 \sinh^2(\pi d) \,.
$$
\end{lemma}

\begin{proof}
When $|c| \to \infty$ then $s \to 0$ and so $g_k \to 0$ so $D(g_0-d^2) = 1$.
Also $g_0 \to 0$ and  inserting these into \cref{eq:evans} and using $\sin(i x) = i \sinh(x)$ 
gives $-2 + 2\cos(2\pi \q) - 4 \sinh^2(\pi \sqrt{d^2}) $ and the result follows.
\end{proof}

Notice that the limiting values of the Evans function at infinity is non-zero unless $\q = d = 0$.

\begin{lemma} \label{lem:real}
If $c\in \R, c \not \in [-1,0)\cup(0,1]$ and if $c \in i \R$, then $E(c)$ is real for $\q,d \in \R$.
\end{lemma}
\begin{proof}
When $c$ is real and $c \not \in [-1,0)\cup(0,1]$ then $Q(\eta)$ is bounded, real, and smooth on $[0,2\pi]$, and so has a real Fourier series, and hence $D(g_0-d^2)$ must be real. All of the other terms in the expression for $E(c)$ are real if $\q$ and $d$ are. 
When $c$ is purely imaginary then by \cref{eq:s} $s$ is also purely imaginary. This implies 
that $\kappa < 0$ is real.
Because of the particular form of the Fourier coefficients $g_k$ Hill's matrix $F$ is Hermitian so that 
Hill's determinant $D(g_0 - d^2)$ is real. 
The argument of $\sin$ in the last term may be either real of purely imaginary, but in either case
the resulting value of $E(c)$ is real.
\end{proof}

The first part of this lemma reflects the fact that if $c$ is real and $c \not \in [-1,0)\cup(0,1]$ then \cref{eq:hill} 
with quasi-periodic boundary conditions represents a self adjoint operator, and so the spectrum must be real. 
The second part of this lemma is related to the fact that when $c$ is purely imaginary then \cref{eq:hill} is PT-symmetric
as discussed in \Cref{subsec:pureimag}.

The Euler equation is an infinite dimensional Hamiltonian system, and hence the spectrum must satisfy the symmetry  
of a Hamiltonian system. This leads to 
\begin{lemma} \label{lem:Ham}
Let $c$ be a root of the Evans function $E(c)=0$ for fixed real $(\q, d)$.  Then also $E(-c) = 0$, $E(\bar c) = 0$, and $E(-\bar c) = 0$.
\end{lemma}
\begin{proof}
{\color{black} Consider the potential $Q$ in  \cref{eq:hill} and note that $Q(\eta, \overline{c} ) = \overline{ Q(\eta, c) }$.
Since $\mu$ is real  \cref{eq:hill}  becomes its conjugate.
Thus, if $y_1, y_2$ is a basis of solutions for the solution space to $y''+Qy = \mu y$, then $\overline{y_1}, \overline{y_2}$ are a basis for $y''+\overline{Q}y = \mu y$. Thus we have that $\Delta(\overline{c}) = \overline{\Delta(c)}$, that is, the Hill discriminant of the original equation \eqref{eq:hill}, with $c \to \overline{c}$ is the complex conjugate of the original Hill discriminant (again provided $\mu \in \R$). This means that $E(\overline{c}) = \overline{E(c)}$ for real $\q$ and $d^2$. In particular, if $E(c) = 0$ then also $E(\overline{c}) = 0$.  
To consider what happens when we change the sign of $c$, note that for the potential $Q$ in \cref{eq:hill} we have that $Q(-\eta,-c) = Q(\eta, c)$. This means that a fundamental set of solutions (normalised to the identity) for $y'' + Q(\eta,-c)y = \mu y$ is $y_1(-\eta,c)$ and $-y_2(-\eta,c)$, where $y_1(\eta,c)$ and $y_2(\eta,c)$ is a fundamental solution set to $\cref{eq:hill}$. In particular, we have that $\Delta(-c) = y_1(2\pi,c) -y_2'(2\pi,c) = y_1(-2\pi,c) + y_2'(-2\pi,c)$. Next we observe that Floquet's theorem says that for any $t \in \R$ $\Phi(\eta + 2 \pi) = \Phi(\eta)\Phi(2 \pi)$, where $\Phi(\eta)$ is a principal fundamental solution matrix to \cref{eq:hill}. Letting $\eta = -2\pi$ gives the result. 
}
\end{proof}

{\color{black} 
Alternatively, note that $s(-c) = -s(c)$ and $s(\bar c) = \overline{s(c) }$. The symmetry under $s \to -s$ can 
be seen directly in the Hill matrix. The factor $\kappa(s)$ is unchanged, hence $g_0$ is unchanged. In the Hill matrix 
every odd off-diagonal entry changes sign under $s \to -s$. This can be compensated by conjugating by 
$\mathrm{diag}( \dots, +1, -1, +1, -1, +1, \dots)$ which leaves Hill's determinant unchanged. 
Conjugating $c$ also conjugates $s$, $\kappa$ and $g_0$.
Further, since $\mu$ is real, in the Hill matrix every odd off-diagonal entry has the opposite sign compared to the complex conjugated matrix entry, so again $\mathrm{diag}( \dots, +1, -1, +1, -1, +1, \dots)$ provides a conjugacy to the original matrix such that
$D(\overline{g_0-d};\overline{s}) = \overline{D(g_0-d;s)}$.
Together with the fact that $\theta \in \R$, this implies $E(\overline{c}) = \overline{E(c)}$ as claimed. }

We want to track the roots of $E(c)$ as we move around in $(\q,d)$ space. 
The circles where $c = 0$ shown in \cref{fig:circles5} divide the parameter plane $(\q, d)$ into 3 types of regions:
\begin{definition}\label{def:regions}
\begin{enumerate}
\item[0:] A point in region 0 is outside all circles
\item[I:] A point in region I is contained inside exactly one circle
\item[II:] A point in region II is contained inside exactly two circles
\end{enumerate}
\end{definition}
In order to show how the spectrum changes when the parameters cross from one region to another 
the derivatives of the Evans function are needed at the boundary circles where $c = 0$.

\begin{lemma}\label{lem:deriv}
The derivative of the Evans function at $c= 0 $ satisfies
\begin{equation*}
\begin{split}
 \pm i \dfrac{\partial E (c = 0) }{\partial c} & = \dfrac{2\pi  \sin (2 \pi \sqrt{1-d^2})}{\sqrt{1-d^2}} 
= \frac{1}{-2d} \dfrac{\partial E (c = 0) }{\partial d}  \, .  %
\end{split}
\end{equation*}
\end{lemma}

\begin{proof} 
Notice that $c = 0$ implies $s = \pm i$. %
In  \cref{lem:trace} we established that $D=1$ at $c =0$ and that any derivative of $D$ at $c = 0$ vanishes.
Thus computing the  $c$- or $d$-derivative of the Evans function boils down to computing 
the derivative of $4 \sin^2 \pi \sqrt{ g_0 - d^2}$. This gives 
\[
     \left. \frac{\partial E}{\partial d}\right|_{c=0} = \frac{ 4 \pi \sin \pi \sqrt{ 1 - d^2}   \cos \pi \sqrt{ 1 - d^2} } { \sqrt{ 1 - d^2}} ( - 2 d)
\]
using $g_0 = 1$ at $c=0$.  
Instead of the $c$-derivative we start by computing the $s$-derivate, using the same observation as before.
The only difference is the final factor which comes from the chain rule, which now is
$$ \dfrac{dg_0}{ds }\bigg{|}_{ \pm i} =  \left. \frac{-4s}{( 1 - s^2)^2} \right|_{\pm i}= \mp i \, .
$$
Using  $\partial c / \partial s|_{\pm i} = 1$  the result follows.
\end{proof}

The two different signs in the $c$-derivative are related to the fact that $c = 0$ is on the branch cut. 
Considering the pre-images under the Joukowski map \cref{eq:s} gives $s = \pm i$.
When approaching $c=0$ from the inside of the unit disk in the $s$-plane we find
\[
    \lim_{\epsilon \to 0^+} c( \pm i (1 - \epsilon) ) = \mp i \epsilon \frac{ 1 - \epsilon/2}{1 - \epsilon}
\]
so that $s = \pm i$ corresponds to approaching the origin $c = 0$ from below (for $s \to +i$) or from above (for $s \to -i$) the branch cut. 

\begin{lemma} \label{lem:normderiv} 
The directional derivative in the normal direction (from inside to out) at the circles $(\q - l)^2 + d^2 = 1$, $l \in \Z$, where $c = 0$  is   
\begin{equation*}
\begin{split}
\nabla_{(\q+l,d)} E(0;\q+l,d)  = -4 \pi \dfrac{\sin(2 \pi \sqrt{1-d^2})}{ \sqrt{1-d^2}} \,.
\end{split}
\end{equation*}
This is positive for $|d| < \sqrt{3}/2$ and negative for $|d| > \sqrt{3}/2$.
\end{lemma}
\begin{proof}
By Lemma \ref{lem:circle}  the circles  are given by $(\q+l)^2+ d^2 = 1$ and the normalised gradient at the point $(\q + l, d)$ 
is $(\q+l, d)$.
Using \cref{lem:deriv} the directional derivative of $E(c,\q,d)$ along the locus where $c=0$ thus is
\begin{equation*}
\begin{split}
\nabla_{(\q+l,d)} E(0,\q+l,d)  & = -4 \pi \left(( \q+l) \sin( 2 \pi (\q+l)) + \dfrac{ d^2  \sin (2 \pi \sqrt{1-d^2})}{\sqrt{1-d^2}} \right) \\
& = -4 \pi \left( \pm \sqrt{1-d^2} \sin(\pm 2 \pi \sqrt{1-d^2}) + \frac{ d^2 \sin (2 \pi \sqrt{1 - d^2})}{\sqrt{1-d^2}} \right) \\
\end{split}
\end{equation*}
and the result follows when putting this on the common denominator.
The $\sin$ function changes sign when $\sqrt{ 1 - d^2} = 1/2$.
In the limit $d \to 1$ the function approaches $-2 \pi$.
\end{proof}
The last lemma is describing the property of the contour-plot of $E$ as a function of $(\q, d)$ shown in \cref{fig:eva} left.
On the upper arc of the circle $\q^2 + d^2 = 1$  where $d > \sqrt{3}/2$ the Evans function $E(c=0; \q, d)$ is decreasing when crossing the circle towards its centre. 
On each of the lower arcs of the circles $(\q \pm 1)^2 + d^2 = 1$ where $d < \sqrt{3}/2$ the Evans function is increasing when 
crossing towards the respective centre.
Now we show that when these circles are crossed, then a pair of pure imaginary roots appears at the origin.
\begin{proposition} \label{pro:cross}
When crossing the circle $(\q+l)^2 + d^2 = 1$, $l \in \Z$, where $c = 0$ from outside to in (i.e towards the centre point)
and  $d \not = 0, \pm \sqrt{3}/2,1$ a pair of purely imaginary roots of $E(c)$ is born out of the real axis at $c =0$.
\end{proposition}
\begin{proof}
Consider the Evans function along a curve $\gamma(t)$ in $(\q, d)$-parameter space 
that crosses the circle $(\q + l)^2 + d^2 = 1$ for $t =0$
in the normal direction with unit speed from the outside to the inside and define $F(c, t) = E(c; \gamma(t))$.
The function $F(c, t)$ defines a curve $c(t)$ through $c = 0$ at $t =0$ if the implicit function theorem holds, 
i.e.\ when the $c$-derivative of $F$  is non-zero at $c = t = 0$. 
According to \cref{lem:deriv} this holds as long as $d \not = 0$ and $d \not = \pm \sqrt{3}/2$.
The derivative of the curve $c(t)$ is given by the implicit function theorem as
\[
   \frac{d c}{ dt} = - \left.\frac{\partial F/\partial t}{\partial F/\partial c}\right|_{0,0} = 
   - \left. \frac{ \nabla_{\gamma'} E }{ \partial E/ \partial c} \right|_{0, \gamma(0)} = \mp i
\]
using \cref{lem:deriv,lem:normderiv}. As before the $\pm$ sign refers to the two sides of the 
branch cut at $c =0$ corresponding to $s = \pm i$. Thus two pure imaginary roots are created.
\end{proof}
The result fails at $d = 0$ and  at $d = \pm \sqrt{3}/2$ which are the points where circles intersect or are tangent.
The propositions can be interpreted as describing a property of the contour-plot of $E$ as a function of $(\q, d)$ shown in \cref{fig:eva} right.
In the right panel $c = \pm 0.1 i$, and the contour-line $E( \pm 0.1 i; \q, d) = 0$ is ``moved'' relative to the nearby circles in
the direction of the inward normal of the circle. 
{\color{black}
The result fails for technical reasons at $d=1$ because then \cref{lem:trace} does not hold.
}

In some sense the last proposition is a bifurcation result where a pair of imaginary roots bifurcates out of the origin. 
Note, however, that unlike in a bifurcation of roots of a polynomial where the derivative of the polynomial 
vanishes, here the derivative is non-zero. 

Having described how a pair of purely  imaginary eigenvalues  $c$ (corresponding to a pair of real $\lambda$, hyperbolic and hence unstable) 
are created the last step is to make sure that they their imaginary part does not vanish. %
Pairs of purely imaginary eigenvalues can collide and form a complex quadruplet.
In principle a complex quadruplet of eigenvalues  may collide on the real $c$-axis and form 
a pair of real eigenvalues (imaginary $\lambda$, elliptic and hence stable). The following lemma shows
that this is impossible.
\begin{lemma} \label{lem:c0}
Let $c$ be real and such that there exists a  quasiperiodic eigenvalue of \cref{eq:hill} with a continuous eigenfunction. Then $c=0$. 
\end{lemma}
\begin{proof}
If $c\in\R$ and $c\in[-1,0)\cup(0,1]$, the potential $Q(\eta)$ is singular and unbounded, and so there are no positive eigenvalues $\mu$ with continuous eigenfunctions here. 
When $c \in \R, c \not \in [-1,1]$  \cref{eq:hill}
is a  Hill's equation with eigenvalue $\mu = 0$.
We have that $g(\eta) = c+ \sin \eta$ is a non-vanishing periodic eigenfunction on $[0,2\pi]$. Because it has no zeros, it is thus 
the eigenfunction corresponding to the largest periodic eigenvalue, and we can conclude that {\em all} of the Hill's spectrum  lies to the left of $d^2= \mu =0$. 
Thus we can not have any bounded solutions for $c\in \R$ with $c \not \in [-1,1],$ with a positive eigenvalue $\mu$ either. 
Thus the only value of $c\in \R$ for which there exists positive quasiperiodic eigenvalues of \cref{eq:hill} is  $c = 0$.  
\end{proof}

\begin{figure}
\includegraphics[width=4cm]{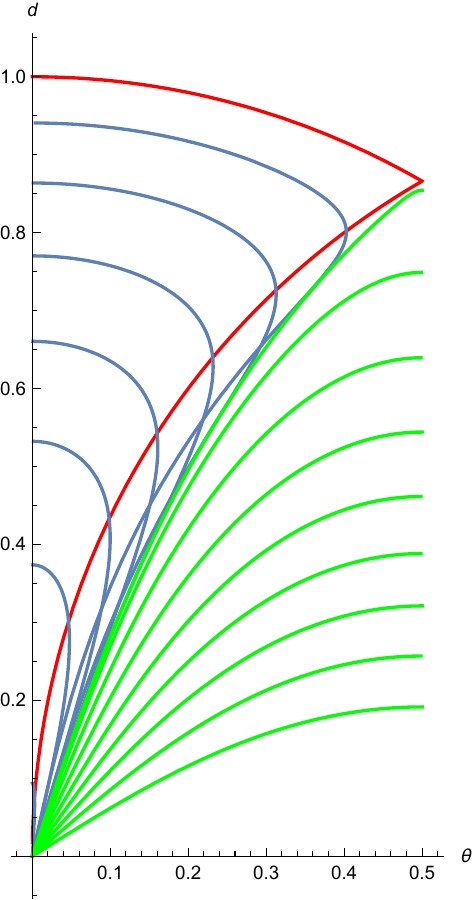}
\caption{Positive quadrant of the scaled unstable disk $\q^2 + d^2 < 1$, $0 \le \q < 1/2$.
Curves of constant  purely imaginary eigenvalues $c/i = 0.1, 0.2, \dots, 0.7$ in blue.
Curves of constant real part of complex eigenvalues $c$ (or $\Re(c)=0.1,0.2,\dots,0.8$) in green, with varying imaginary part of $c$.
Region  I with a purely imaginary pair of eigenvalues is bounded by the red line. 
Below the red line in region II there are either two imaginary pairs of eigenvalues, indicated by intersecting blue contours,
or there is a complex quadruplet as indicated by the green contours.
} \label{fig:diskev}
\end{figure}

We are now ready to state the main result of this section. We have already established that 
there is an eigenvalue $c=0$ on the circles in \cref{lem:circle}.
Whenever two circles intersect the multiplicity doubles. 
In the following theorem we are counting all other isolated eigenvalues. 
\begin{theorem} \label{thm:evanscount}
Consider the isolated eigenvalues $c$ as given by the zeroes of the Evans function $E(c; \q, d)$ other than $c = 0$. %
When $(\q, d)$ is in region 0 there are no eigenvalues.
When $(\q, d)$ is in region I there is a pair of purely imaginary eigenvalues. 
When $(\q, d)$ is in region II there are 4 eigenvalues, either two purely imaginary pairs, a single pair with multiplicity 2, or a complex quadruplet. %
On the boundary between region I and region II there is a pair of purely imaginary eigenvalues.
For all cases half of the eigenvalues are unstable. 
\end{theorem}
\begin{proof}
We have been considering the quasiperiodic eigenvalue problem  \cref{eq:hill} with parameter $c \in \C$
and are looking for real positive eigenvalues $\mu = d^2$ with Floquet exponent $2 \pi \q$. 
We have constructed an Evans function $E(c;\q,d)$ viewed as a function of $c$ with  parameters $\q$ and $d$. 
By Li's unstable disk theorem \cite{Li00}, see above \cref{thm:disk}, there are no isolated eigenvalues in region 0.
This implies that the Evans function $E(c)$ has no zeroes when the parameters $(\q, d)$ are in region 0. 
Consider a path of parameters $(\q, d)$ from region 0 to region I that crosses 
the boundary circle away from $d = \sqrt{3}/2$.  According to \cref{pro:cross} this  creates a pair of purely imaginary roots.
Hence there are 2 isolated eigenvalues in region I. Now consider a path of parameters from region I to region II. 
Crossing the boundary circle away from $d =\sqrt{3}/2$ and $d = 0$ another pair of purely imaginary roots is created. 
Hence there are 4 isolated eigenvalues in region II. Should the two pairs of eigenvalues coincide they are counted with 
multiplicity. This does in fact occur, and leads to a complex quadruplet of eigenvalues. Since we are counting 
with multiplicity there are always 4 isolated eigenvalues in region II. 
Because of \cref{lem:c0} it is not possible that the quadruplet of eigenvalues bifurcates onto the real axis
and by Hamiltonian symmetry of the spectrum half of the eigenvalues correspond to unstable eigenvalues.
\end{proof}

To conclude this section we present some numerical results that show how the eigenvalues change in the fundamental 
domain in $(\q, d)$-space.
Throughout we have considered $(\q, d)$ as continuous parameters, such that the eigenvalues $c$ 
become functions of these parameters. 
For the case of purely imaginary $c$ curves of constant $c$ are plotted in the $(\q, d)$-plane, see \cref{fig:diskev}.
The blue curves are given in parametric form by $(\q(d), d)$ where $\q(d)$ is determined by $E(c; \q, d) = 0$, 
or equivalently $\Delta(d^2;c) = 2 \cos (2 \pi \q)$ for constant purely imaginary value of $c$. 
For the green curves consider $c = c_r + i c_i$ for constant real part $c_r$ and the imaginary part 
$c_i$ as a parameter. For such given $c$ find $d$ such that $\Im \Delta(d^2; c_r + i c_i)  = 0$.
Then determine $\q$ from $\Delta(d^2; c) = 2 \cos (2 \pi \q)$ as before and thus 
obtain the green curve  $(\q(d(c_i)), d(c_i))$ in parametric form.
The extremal values of the curve parameter $c_i$ are determined by $\Delta = \pm 2$.
These occur at the origin in $(\q, d)$ and at $\q = 1/2$, respectively. 

\section{Back to Euler's equations} \label{sec:euler}

So far we have constructed an Evans function $E(c)$ for given parameters $(\q,d)$, determined  
by the wave number $k$ and the overall integer parameter vector $\bp$.
In the original Euler equations these are related to an integer mode number vector $\ba = k \bq + l \bp =  d \bp_\perp + \q \bp$. 
When $\ba$ is inside the unstable disk corresponding perturbations have eigenvalues $\lambda=-ikc$ with non-zero real part.
To find the unstable spectrum of the Euler equations linearised about $\cos( p_1 x  + p_2 y)$ the eigenvalues 
corresponding to all parameters $(\q, d)$ 
inside the unstable disk need to be combined into a single Evans function. 
In the previous section it was convenient to define the Evans function $E(c)$ to be a function of $c = i \lambda / k$.
Now putting them together we return to the original temporal spectral parameter $\lambda$.
\begin{theorem} \label{thm:evansp}
An Evans function for the 2-dimensional Euler equations on the torus with coordinates $(x,y)$ 
linearised about the  { \color{black} steady state $\cos( x p_1 + y p_2)$ for 
 integers $p_1$, $p_2$ with $\gcd( p_1, p_2)=1$ }
 is given by
\[
   E_{\bp} (\lambda) = \prod_{k=1}^{p^2-1} E( i \lambda/k;  k \bp \cdot \bq / p^2, k / p^2)^2 
\]
where $p = \sqrt{ p_1^2 + p_2^2}$.
\end{theorem}
\begin{proof}
The wave number $k$ is an integer giving the periodicity of the perturbation in the $\xi$-direction.
It determines the parameters 
\[
         (\q(k), d(k))  = ( k \bq \cdot \bp/p^2 \bmod 1, k/p^2) \,.
\]
 As noted before $\q$ is only defined $\mod 1$. 
We choose a fundamental domain so that $\q \in (-1/2, 1/2]$.
\footnote{Even though the Evans function is even in $\q$, the set of points $(\q(k), d(k))$ does not possess that 
symmetry. Thus for counting lattice points the domain is $\q \in (-1/2, 1/2]$, while
for considering the values of the Evans function $\q \in [0, 1/2]$ is enough.}
By Li's unstable disk theorem \cite{Li00}, see \cref{thm:disk}, there are no unstable eigenvalues 
when $|k| > p^2$. 
When $|k| = p^2$ we are on  the boundary of the unstable disk when $\q(k) = 0$, and outside otherwise. 
When we are on the boundary of the unstable disk according to \cref{lem:circle} we have $c = 0$, and hence no eigenvalue
$\lambda$ with non-zero real part. 
To capture the point spectrum we can thus restrict to $|k| < p^2$.
Note that for large $|k|$ the corresponding point $(\q(k), d(k))$ may be outside (or on the boundary of) the unstable disk.

When $k=0$ \cref{eq:sep} becomes the trivial equation $f'' = 0$ which has no periodic 
solutions other than the constant solution. 
So we can restrict to non-zero waver numbers $k$ with $|k| < p^2$. 
Since $E(c; \q, d)$ in \cref{eq:evans} is  even in $\q$ and in $d$ we can restrict to positive $k$
and instead square the Evans function. This shows that in $E_{\bp}$ every eigenvalue has at least multiplicity 2. 
So we end up with a product of $E(c; \q(k), d(k))^2$ over $k$ from 1 to $p^2 - 1$ where $(\q, d)$ are functions
of $k$. Finally also $c$ is a function of $k$, since the Evans function of the full Euler equations
instead of the separated ODE is a function of $\lambda = -i k c$.

Since the point spectrum has only  finitely many elements there is no convergence problem with the product.
Nevertheless, we note that using \cref{lem:infinity} it is possible to normalise each Evans function such 
that its absolute value approaches 1 at infinity.
\end{proof}

Evans functions are only defined up to a non-zero multiplicative constant, so more, or fewer, terms 
could have been included in the product in \Cref{thm:evansp}. An obvious choice for more terms would have been to include all $k \in \Z$, while an obvious choice for fewer terms would be to remove any $k$ such that $(\q(k),d(k))$ lie outside the unstable disk (see the leftmost panel of \cref{fig:circles}). Each of these choices would have produced an Evans function with the same roots as the one in \Cref{thm:evansp}. However, each of these choices has its own computational disadvantages. For example, choosing all $k \in \Z$ requires one to normalise the Evans function for each class, in both $\lambda $, and $k$, while choosing only the terms inside the unstable disk is cumbersome (and different) for different choices of $\bp$, and so becomes convoluted to compute in practice. We have chosen this particular set of values of $k$ because it  produces an Evans function that is readily computable, while still showcasing the importance of the unstable disk (and the lack of contribution from terms outside it). 

In \cite[Theorem 3]{LLS04} it was proved that the number of non-imaginary isolated eigenvalues 
does not exceed twice the number of integer lattice points inside the unstable disk, not counting 
any lattice points on the line through $\bp$ (in particular not counting the origin).
Numerical experiments in \cite{Li00} and also in \cite{DMW16} indicate that this upper bound is sharp.

For example for $\bp = (4,5)$ there are 128 lattice points in the open unstable disk of radius $\sqrt{ 41}$, not counting the origin.
In this case let $k = 1, \dots, 40$,  and 11 times  $(\q(k), d(k))$  lies in region I,  26 times in region II, once (for $k=9$) 
on the  inner boundary  between region I and II,
once (for $k=39$) on the boundary of the unstable disk, and once (for $k=40$) in region 0 (outside the unstable disk).
On the inner boundary a pair is born at the origin, but at that moment it is not yet an unstable eigenvalue. 
Thus there are $(11+1) \cdot 2 + 26 \cdot 4 = 128$ distinct eigenvalues with non-zero real part. 
The  lattice points for $k = 39$ and $k=40$ on and outside the unstable disk produce an Evans function 
factor that does not have any eigenvalues with non-zero imaginary part, in fact it is non-zero everywhere
in $\C \setminus i [-1, 1]$.
The 40 wave numbers are only half of the wave numbers that may have $(\q, d)$  inside the unstable disk,
the other half are obtained for negative $k$.
They sit in the opposite half of the circle, relative to the line through the lattice point $\bp$.
See \cite{DMW16} for a few more examples of this type.
Counting the number of lattice points inside a circle around the origin is a hard problem. 
But we don't have to actually count this number, we just have to show that the Evans function $E_{\bp}(\lambda)$
has this number of zeroes. More precisely we have to establish that for each factor with 2 or 4 roots
according to \cref{thm:evanscount} there are 2 or 4 lattice points inside the unstable disk. 
Combining all of the results about the Evans function we can now prove our main theorem.
\begin{theorem}\label{thm:sharp}
The number of  eigenvalues $\lambda$ with non-zero real part for the 2-dimensional Euler equations
on the torus linearised about the steady state with stream function $\cos( p_1 x + p_2 y)$ 
{\color{black}  for integers $p_1, p_2$ with $\gcd(p_1, p_2) = 1$}
is equal to twice the number of non-zero integer lattice points 
inside the disk of radius $\sqrt{ p_1^2 + p_2^2 }$.
\end{theorem}
\begin{proof}
By construction lattice points $(\q(k), d(k))$ in the  
product that defines $E_{\bp}$ are inside the fundamental domain $\q \in (-1/2, 1/2]$.
The corresponding lattice points is $\ba = d \bp_\perp + \q \bp$ which may or may 
not be inside the unstable disk defined by $\ba^2 < \bp^2$.
Now consider the number of lattice points $\ba = k \bq + l \bp$ for given $k$ and arbitrary $l$
that are inside the unstable disk. We are going to show that when 
$(\q(k), d(l))$ is in region I that number is 1, while when 
$(\q(k), d(l))$ is in region II that number is 2.
The main observation is that \cref{fig:circles} and the regions it defines 
are translation invariant by $\pm1$ in the $q$ direction.
The unstable disk is the disk in the centre. Any point in region I in the unstable 
disk shifted by $\pm 1$ is again in region I, but outside the unstable disk.
Thus for every  $(\q(k), d(k))$ in region I that point is the only lattice point in the unstable disk.
For region II the situation is more interesting. 
Consider a lattice point $(\q(k), d(k))$ (corresponding to $\ba$) inside region II in the unstable disk, say with $\q < 0$. 
Shifting this lattice point by 1 to the right (corresponding to  $\ba + \bp$) it is again in region II, and still inside the unstable disk.
Similarly for $\q > 0$ and a shift by 1 to the left (corresponding to $\ba - \bp$). 
Thus for every lattice point $(\q(k), d(k))$ inside region II in the unstable disk there 
are two lattice points inside the unstable disk.

Now according to  \cref{thm:evanscount}  for $(\q,d)$ in region I the Evans function 
has two roots, while for $(\q, d)$ in region II the Evans function has four roots. 
Combining this with the observation about lattice points just made we see
that each lattice point accounts for two roots. 

The exceptional case is when $(\q(k), d(k))$ is on the boundary between region I and region II.
As shown in \cref{thm:evanscount} this factor has two eigenvalues with non-zero real part. 
For this point $(\q \pm 1, d)$ (corresponding to $\ba \pm \bp$) is either outside or on the boundary of the unstable disk, and 
hence again the number of eigenvalues  is twice the number of lattice points.
Any points $(\q(k), d(k))$ in region 0 are outside the unstable disk and the corresponding factor has no eigenvalues.
Any points $(\q(k), d(k))$ in the boundary between region 0 and region 1 are not lattice 
points inside the unstable disk and only have eigenvalue $\lambda = 0$, and hence
have no eigenvalue with non-zero real part. This exhausts all possible cases. 

Finally, for every point $(\q, d)$  in the fundamental domain inside the unstable disk 
there is another point  $(-\q, -d)$ inside the unstable disk.
The Evans function $E(c; \q, d)$ is even in $\q$ and $d$, and this is why only positive $k$
are included, but each factor is squared. 
This doubles the number of lattice points, but also doubles the number of roots. 
Thus altogether the number of eigenvalues as given by roots of $E_{\bp}( \lambda)$ 
is equal to twice the number of non-zero lattice points inside the unstable disk.
\end{proof}

\section{Summary and Discussion}
We have constructed an Evans function for the linearised operator about shear flow solutions to the Euler equations of the form $\psi = \cos(p_1x+p_2y)$ for relatively prime $p_1$ and $p_2$, and have proven the sharpness of the upper bound of the number of points in the point spectrum found in \cite{LLS04}, that is the number of discrete eigenvalues of the operator is exactly twice the number of non-zero integer lattice points inside the so-called unstable disk. To compute the Evans function we separated variables in the linearisation and transformed the dispersion relation of a class into an equation of Hill's type. We were then able to use properties of the Hill determinant to determine that points of spectrum for this class coincided with integer lattice points inside the unstable disk. Finally, we multiplied each of the relevant Evans functions for each class intersecting the unstable disk, resulting in the function given in the Evans function for the full 2+1 dimensional linear operator given in \cref{thm:evansp}. 

We note that in \cref{thm:evansp,thm:sharp}, we explicitly require $\gcd(p_1,p_2) =1$. The reason is that otherwise we cannot conclude that there is a uni-modular transformation to our new coordinates in \cref{sec:separate}. In particular, we are not allowed to treat $\bp = (0,2)$. We expect that the computation can be repeated keeping the gcd as a factor so that the general statment from \cite{LLS04} for non co-prime $\bp$ can be recovered. 

Lastly, we note that while consider the original shear flow and its perturbations as being on the torus, we could have considered them on $\R^2$, and then the perturbations we have considered are those co-periodic in both $x$ and $y$. Relaxing the conditions of co-periodicity in $x$ for example, we would get an additional parameter corresponding to the Floquet exponent of the (now only quasi-periodic) boundary conditions of \cref{eq:sep}, which would then filter through the calculations, effectively shifting the points of spectrum in the complex plane. Relaxing the co-periodicity condition in $y$ would do the same, only now the shifting parameter is the (now continuous) parameter $k$. It remains to be seen which curves and/or surfaces would appear as the points of spectrum are plotted in terms of these additional parameters, as well as their relation to points in and on an ``unstable disk", which to the best of our knowledge, has yet to be defined for such boundary conditions. 

\section{Acknowledgements} 
The authors would like to thank G.~Vasil for very useful discussions.
 RM would like to acknowledge the support of Australian Research Council under grant DP200102130. 
 HRD and RM acknowledge  support through the Australian Research Council  grant DP210101102.

\appendix

\section{Relation to Jacobi Operators}

In  previous papers \cite{Li00,LLS04,DMW16} the analysis of the spectrum of the 2-dimensional Euler 
equations on the torus around a shear flow was done by analysing the operator in its representation 
in Fourier space, which after block-decomposition leads to so called classes which are represented 
by Jacobi operators. Let $a_j$ be the Fourier mode corresponding to the lattice point $\ba + j \bp \in \Z^2$.
Then the Jacobi operator of the class lead by $\ba$ is given by the difference equation
\[
      R(j) ( a_{j+1} - a_{j-1} ) = \lambda a_j , \quad R(j) = \||\bp||^{-2} - ||\ba + j \bp||^{-2}\,,
\]
where $\lambda$ is the spectral parameter.
Since $R(j) = P(j)/Q(j)$ is rational in $j$ the recursion relation can be written in polynomial form as
\[
      P(j) ( a_{j+1} - a_{j-1} ) - \lambda Q(j) a_j = 0 \,.
\]
This can be considered as coming from a Fourier transform, and to identify this we 
relabel the coefficients such that 
\[
      P^+(j+1) a_{j+1} + P^-(j-1) a_{j-1} - \lambda Q(j) a_k = 0 \,.
\]

Let's start at the other end and consider a generalisation of a Hill's equation in the form 
\[
       A u'' + B u' + C u = 0
\]
where $A,B,C$ are periodic functions. In the simple case at hand they are just linear combinations 
of $\sin$ and $\cos$. Now assuming a Fourier series solution $u(\eta) = \sum a_j e^{i j \eta}$
and inserting this into the ODE and collecting coefficients in terms of $e^{i j\eta}$ gives
a recursion relation for the coefficients $a_j$. We obtain
\[
     \sum ( -A j^2  + i B j + C) e^{ij\eta} = 0
\]
Instead of derivatives we have multiplication by powers of $ij$. When multiplying the term $a_j e^{ij\eta}$ 
in the Fourier series by a term $ e^{il\eta}$ of a periodic coefficient the product is $a_j e^{i (j+l) \eta}$. 
To get the recursion relation for the $a_j$ all exponential terms with a shifted exponent have to be relabelled to
$a_{j-l} e^{ij\eta}$.

Since the Fourier transform turns differentiation into multiplication we see that
the number of Fourier modes in the coefficients $A,B,C$ of the ODE determines the order of the DE,
while the degree of the coefficients $P,Q$ in the DE determines the number of derivatives in the ODE.

Note that the ODE can always be transformed into Hill's form (different $Q$!)
\[
   u'' + Q u = 0
\]
where $A Q = C - B^2/(4 A) - B'/2 + B A'/(2 A) $. However, this transformation 
changes the boundary conditions, which is why we need to consider quasi-periodic boundary 
conditions for a Hill's equation in standard form.

\section{Counting Modes}

We define the following map on classes. The class corresponding to the line $y = \frac{p_2}{p_1}x + \frac{k}{p_2} $ is mapped to the point $(\q(k), d(k))$ in the square $\cS = [-\frac12,\frac12] \times [-1,1]$ given by 
$$
\q(k) = k \frac{\bp \cdot \bq}{p^2} \mod{1}   \quad \quad d(k) =  \frac{k}{p^2}. 
$$

The $\q$ coordinate is in the right range by construction, while the $d$ coordinate can be seen to be in the right class by substituting the equation for the class
$$y = \frac{p_2}{p_1}x + \frac{k}{p_2} $$
into the equation for the unstable disk
$$ x^2 +y^2 \leq p^2$$ 
and observing that points inside the unstable disk on the class have $x \in (r_1, r_2)$ the roots of the corresponding quadratic. The discriminant of this quadratic is $\sqrt{1-\frac{k^2}{p^4}} = \sqrt{1-d^2}$ and so any (real) roots  has $d \in [-1,1]$. 

For calculations it is easier to consider instead $\q \in [0,1]$. This is of course the same as $\q \in [-1/2, 1/2]$ because of the lemma where the circles come from. The circles describe three regions in $(\q,d)$ space which are described in \cref{fig:circles}. Now instead we consider two circles in $(\q,d)$ space, one centred at $(0,0)$ and one centred at $(1,0)$. \Cref{lem:circle} means that points on the unstable disk get mapped to points on the circles $\q^2+d^2=1$ and $(\q \pm 1)^2 + d^2 =1$ in $(\q,d)$ space. 

We have the following three lemmata

\begin{lemma}\label{lem:outside} 
Suppose that $\q(k)^2 + d(k)^2 > 1 $ and that  $(\q(k)-1)^2 + d(k)^2 > 1 $. Then there are no integer lattice points on the line $y = \frac{p_2}{p_1}x + \frac{k}{p_1}$ inside the unstable disk. 
\end{lemma}

\begin{lemma}\label{lem:lemma2pts}
Suppose that $\q(k)^2 + d(k)^2 < 1 $ and that  $(\q(k)-1)^2 + d(k)^2 < 1 $. Then there are two integer lattice points on the line $y = \frac{p_2}{p_1}x + \frac{k}{p_1}$ inside the unstable disk. 
\end{lemma}

\begin{lemma}\label{lem:lemma1pt}
Suppose that either  $\q(k)^2 + d(k)^2 > 1 $ and that $(\q(k)-1)^2 + d(k)^2 < 1 $ or  $\q(k)^2 + d(k)^2 < 1 $ and that  $(\q(k)-1)^2 + d(k)^2 >1 $. Then there is a single integer lattice point on the line $y = \frac{p_2}{p_1}x + \frac{k}{p_1}$ inside the unstable disk. 
\end{lemma}

\begin{proof}

All three lemmas can be proved at once, by computing the intersection of the line of a class with the unstable disk in the coordinates set by $\bp$ and $\bq$. That is we compute 
$$
\begin{pmatrix} m \\ k \end{pmatrix} = \begin{pmatrix} q_2 & -q_1 \\ -p_2 & p_1 \end{pmatrix} \begin{pmatrix} x_\pm \\ y_\pm \end{pmatrix}
$$
where $x_\pm$ and $y_\pm$ are the intersection points of the class $y = \frac{p_2}{p_1}x + \frac{k}{p_1}$ with the unstable disk $x^2+y^2 = p^2$. This gives 
$$
 \begin{pmatrix} m_\pm \\ k \end{pmatrix}= \begin{pmatrix} \frac{-k \bp \cdot \bq}{p^2} \pm \sqrt{1-\frac{k^2}{p^4}} \\ k \end{pmatrix} 
$$
and so now the only thing left is to determine how many integers lie between $m_\pm$ for a fixed $k$, but this is exactly what is described in the statement of the lemmata. That is if $\q(k)^2 +d^2 >1$ and $(\q(k)+1)^2 +d^2 >1$ then there are no integers between $m_+$ and $m_-$ and hence no integer lattice points of the class $y = \frac{p_2}{p_1}x + \frac{k}{p1}$ inside the unstable disk, \& tc. 

Expanding on this a bit, this is the number of integer points between $a \pm \sqrt{1-b}$ (inclusive). This can be thought of as the number of integers $n$ such that $(n-(a+\sqrt{1-b^2}))(n-(a-\sqrt{1-b^2})) = n^2 - 2an +(a^2+b^2-1) \leq 0$ with $a \in [- \frac12,\frac12]$ and $b \in [-1,1]$. This has an upper bound of $3$ with $n = \pm 1, 0$ and $a = b = 0$.  The lower bound is clearly $0$, which is achieved when $b = \pm 1$, but more than that, if $a^2+ b^2 >1 $ and $(a+1)^2+b^2 >1$ then the roots of the quadratic lie in the interval $(0,1)$, and we have exactly the statement of \cref{lem:outside}. The other two statements can be similarly arranged.

\end{proof}

\begin{figure}
\includegraphics[scale=0.22]{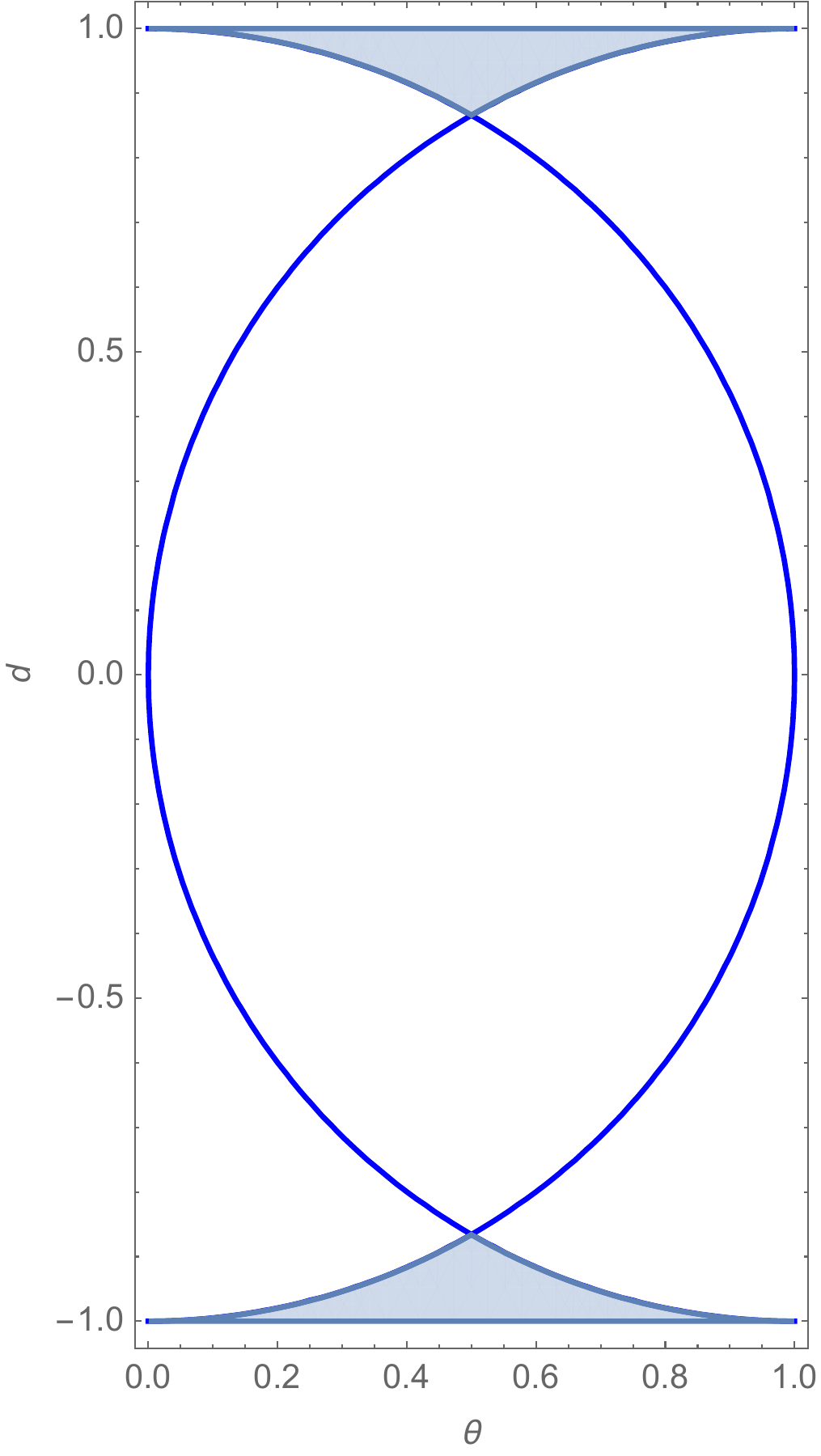}
\includegraphics[scale=0.22]{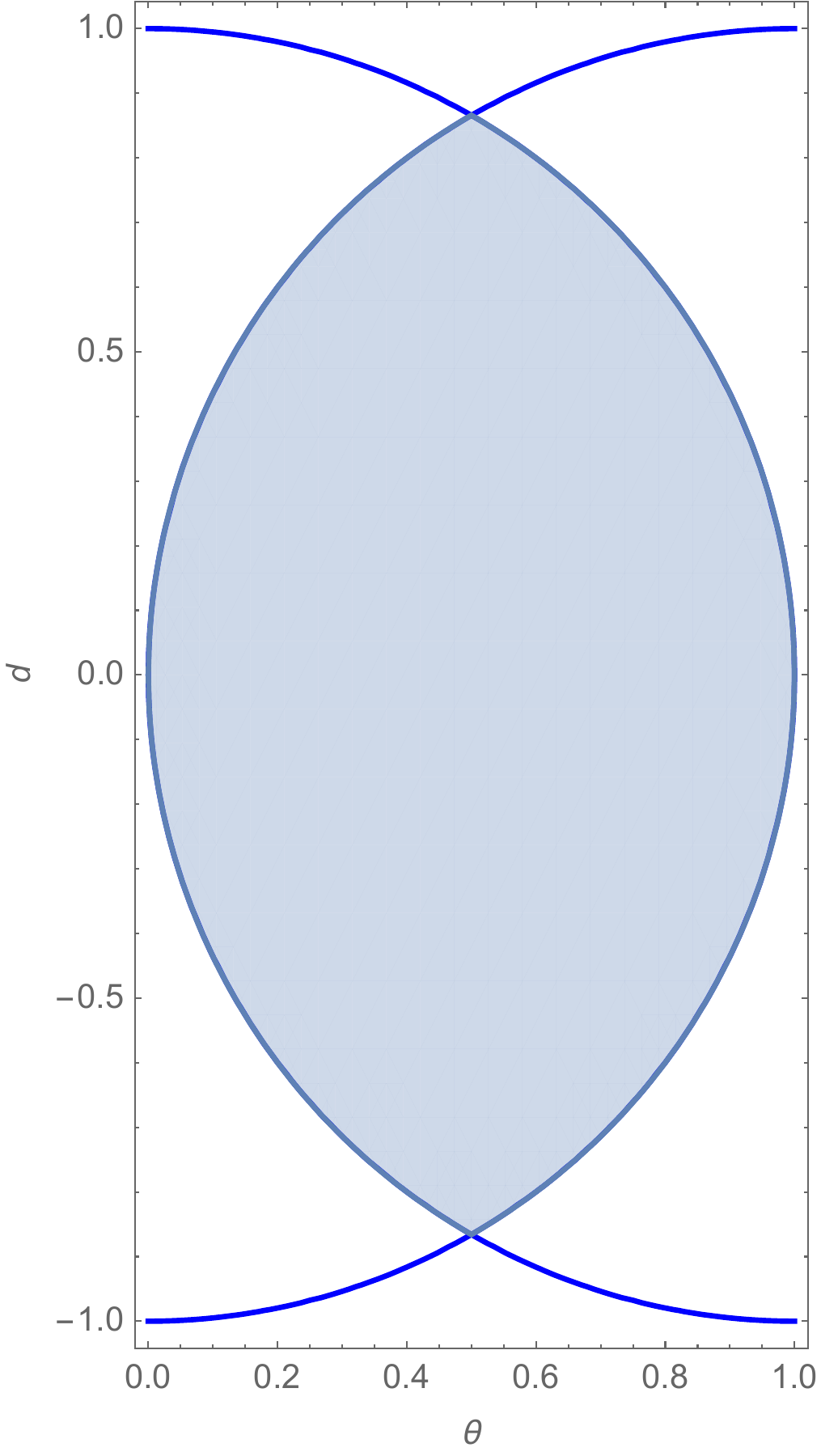}
\includegraphics[scale=0.22]{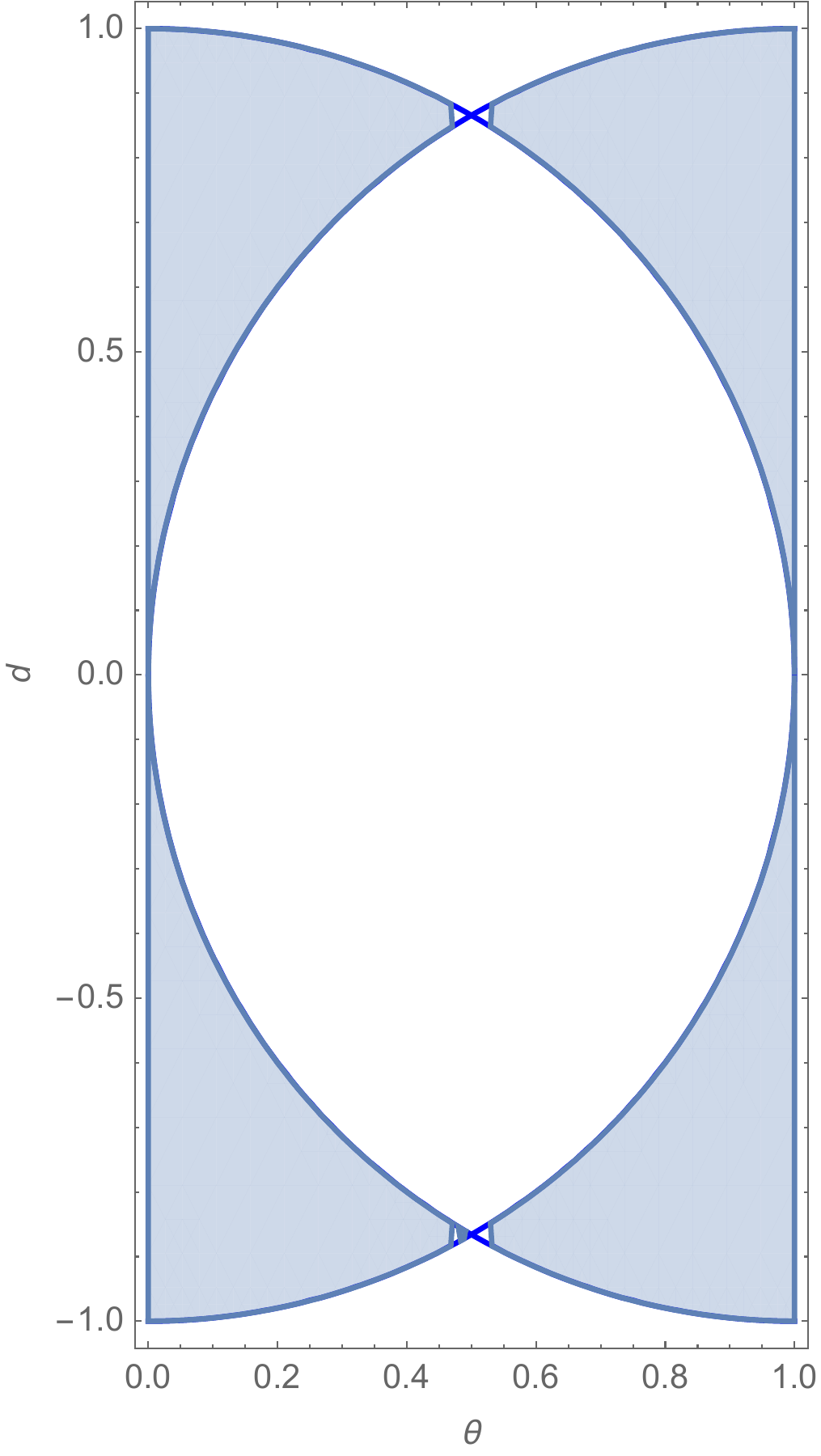}
\caption{The regions in $(\q,d)$ space (from left to right) \cref{lem:outside,lem:lemma2pts,lem:lemma1pt} respectively.} \label{fig:circles}
\end{figure}

\bibliographystyle{alpha}      
\bibliography{all.bib,hd.bib}  

\end{document}